\documentclass[11pt]{article} 
\usepackage{a4,amssymb, amsmath, amsthm}
\usepackage{graphics,color}
\usepackage{epsfig}

\DeclareUnicodeCharacter{0301}{\'{e}}
      
\newtheorem{theorem}{Theorem}[section]

\newtheorem{proposition}[theorem]{Proposition}

\theoremstyle{definition}
\newtheorem{definition}[theorem]{Definition}
\theoremstyle{remark}

\theoremstyle{rem}

\newcommand{\R}{\mathbb{R}}
\newcommand{\N}{\mathbb{N}}
\newcommand{\E}{\mathbb{E}}

\newcommand{\sol}{\phi}
\newcommand{\secref}[1]{Section~\ref{#1}}

\newcommand{\figref}[1]{Figure~\ref{#1}}
\newcommand{\tabref}[1]{Table~\ref{#1}}
\newcommand{\randomvar}{\xi}
\newcommand{\stochdim}{m}
\newcommand{\probmeasure}{\pi}
\newcommand{\polyapproxspace}{\mathcal{W}_J(\Xi)}
\newcommand{\SGOrder}{J}
\usepackage{todonotes}
		
\usepackage{caption}
\usepackage{subcaption}
\usepackage{tikz}
\usepackage{pgfplots}
\pgfplotsset{compat = newest}
\pgfplotscreateplotcyclelist{color}{%
blue,thick,,every mark/.append style={fill=blue!80!black},mark=*\\%
red,thick,,every mark/.append style={fill=red!80!black},mark=square*\\%
brown!60!black,thick,,every mark/.append style={fill=brown!80!black},mark=otimes*\\%
blue,thick,,densely dashed,every mark/.append style={fill=blue!80!black,solid},mark=*\\%
red,thick,,densely dashed,every mark/.append style={fill=red!80!black,solid},mark=square*\\%
brown!60!black,thick,,densely dashed,every mark/.append style={fill=brown!80!black,solid},mark=otimes*\\%
}
\begin{document}

\providecommand{\keywords}[1]{{\noindent \textit{Key words:}} #1}
\providecommand{\msc}[1]{{\noindent \textit{Mathematics Subject Classification:}} #1}

\title{Space-time stochastic Galerkin boundary elements for acoustic scattering problems\\\vskip 0.8cm}
\author{Heiko Gimperlein\thanks{University of Innsbruck, Engineering Mathematics, Technikerstra\ss e 13, 6020 Innsbruck,
Austria, e-mail: heiko.gimperlein@uibk.ac.at.}, Fabian Meyer\thanks{Institute of Applied Analysis and Numerical Simulation, University of Stuttgart, Pfaffenwaldring 57, 70569 Stuttgart, Germany.}, Ceyhun \"{O}zdemir${}^{\ast}$\thanks{Institute for Mechanics and Graz Center of Computational Engineering, Graz University of Technology, 8010 Graz, Austria, e-mail: ceyhun.oezdemir@uibk.ac.at. \\ $\ $ \\There are no conflicts of interest related to this work.}}

\maketitle \vskip 0.5cm
\begin{abstract}
\noindent Acoustic emission or scattering problems naturally involve uncertainties about the sound sources or boundary conditions. This article initiates the study of time domain boundary elements for such stochastic boundary problems for the acoustic wave equation.  We present a space-time stochastic Galerkin boundary element method which is applied to sound-hard, sound-soft and absorbing scatterers. Uncertainties in both the sources and the boundary conditions are considered using a polynomial chaos expansion. The numerical experiments illustrate the performance and convergence of the proposed method in model problems {and present an application to a problem from traffic noise}.  
\end{abstract}


\section{Introduction}

This article introduces a boundary element method for the quantification of uncertainties in time-dependent sound emission and scattering problems. In these problems the nature of the sound source and the reflection or absorbing properties of the scatterer is often difficult to measure, for example, in traffic noise to assess the sound emission of vehicles \cite{comput}.  

The quantification of uncertainties is by now a key consideration in the assessment of numerical models in engineering    \cite{Chen20211,Freschi202218113,Ranftl202058,Ryan2021713}. We refer to  \cite{BFT07,BTZ04,BespCathSil12,Charr12,Deb01,Frau05} for relevant works on partial differential equations involving random variables, in time-dependent or time-independent situations. 
{The survey \cite{gunzburger14} discusses modern stochastic collocation and Galerkin approaches, especially for time-independent problems. Specifically for the wave equation in a bounded domain, reference \cite{Motamed_2012} considers a stochastic collocation method for problems with an uncertain wave velocity, but deterministic initial and boundary conditions. Extensions to elastodynamics are considered in \cite{Motamed_2015}. Also geometric uncertainties of the domain have recently been considered in the frequency domain \cite{henriquez2021diss,henriquez2021shape}. } 

In this article we propose and analyze a stochastic space-time Galerkin boundary element method for the acoustic wave equation {in unbounded,  exterior domains} with uncertain boundary conditions, based on a polynomial chaos expansion \cite{Wiener1938,XiuKarniadakis2002,fm19} of the random variables. Our approach is based on recent advances for time domain boundary elements, which reduce the {exterior} wave equation to an integral equation on the boundary. Such methods are particularly relevant for emission or scattering problems, which involve wave propagation over large distances. Both space-time Galerkin and convolution quadrature methods have been developed, see \cite{costabel04,hd,sayas2016retarded} for an overview.  Recent developments in the directions of the current work include stable formulations \cite{Aimi09,LalSayas09,SteinTorres22}, efficient discretizations \cite{aimi2023higher,desiderio2020efficient,graded,gimperlein2019hp,residual,PoelzSch21}, compression of the dense matrices \cite{aimi2022partially,seibel2022boundary} as well as complex coupled and interface problems \cite{aimi2023time,aimi2024weak,Aimi14,desiderio2023cvem,contact,FSI,wave-wave-coupling,Hsiao19}.

To be specific, the method presented here reduces stochastic boundary value problems for the {exterior} acoustic wave equation to integral equations on the boundary of the computational domain. Using a polynomial chaos expansion of the random variables, a high-dimensional integral equation is obtained which is then discretized by a Galerkin method in space and time. The resulting high-dimensional, but sparse, space-time system is solved in a marching-on-in-time time-stepping scheme \cite{terrasse}.
The approach is presented for stochastic Dirichlet, Neumann and acoustic boundary conditions, with uncertain sources and acoustic parameters. We develop the theoretical background to establish stability and convergence of the proposed discretization of the stochastic boundary integral equations. The numerical experiments in  three space dimensions illustrate the performance and accuracy of this method for increasing stochastic, respectively space-time degrees of freedom. {An application to a problem from traffic noise is presented.}

{Note that the size of the linear algebraic system for stochastic Galerkin finite elements in three space dimensions remains challenging, in spite of extensive research (see p.~566 in \cite{gunzburger14} for a review of relevant works). Our use of boundary element methods for the wave equation crucially leads to a \emph{sparse} system of equations in \emph{two} space dimensions.}\\

\noindent This article is organized as follows:
In Chapter \ref{stochprob}, we introduce stochastic Dirichlet, Neumann and acoustic  boundary value problems for the wave equation. We introduce the relevant boundary integral operators for the wave equation and formulate the boundary value problems as integral equations on the boundary. Furthermore we discuss abstract stochastic space-time Galerkin methods for the stochastic Dirichlet, stochastic Neumann and stochastic acoustic problems. Their stability follows from the weak coercivity of the formulation. 
In Chapter \ref{sec:SSTBEM}, we use a polynomial chaos expansion and discretize the Galerkin formulation using orthonormal basis functions in the stochastic variables. The discretization in space and time uses a tensor-product ansatz, leading to a marching-on-in time scheme.
Chapter \ref{sec:numericalExpetiments} presents the numerical experiments for uncertain sources and acoustic parameters, {both in model problems and for a standard problem in traffic noise.} Details of the discretization and the theoretical framework of Sobolev spaces are presented in appendices.\\

\noindent {\emph{Notation:} In this article we denote the time derivative of a function $f$ by $\partial_t f$ or also by $\dot{f}$. We further write $f \lesssim g$ provided there exists a constant $C$ such that $f \leq Cg$. If the constant
$C$ is allowed to depend on a parameter $\sigma$, we write $f \lesssim_\sigma g$.}

\section{Stochastic boundary problems and integral formulations}\label{stochprob}

We consider the  {exterior} wave equation
\begin{align}\label{eq:waveeq}
\partial_t^2 u - \Delta u = 0\ , \quad u=0\ \text{ for }\ t\leq 0\ 
\end{align}
in the complement $\mathbb{R}^d \setminus \overline{\mathcal{X}}$ of a polyhedral domain or screen,
with $\overline{\mathcal{X}} \subset \R^d$,
where we focus on the cases $d=2$ and $d=3$. 
On the boundary $\Gamma = \partial \mathcal{X}$ we either prescribe stochastic Dirichlet (SD), stochastic Neumann (SN) or 
stochastic acoustic boundary conditions (SA), involving a random right hand side $f$
\begin{align}
u(t,x,\randomvar) &= f(t,x,\randomvar) \tag {SD} \label{eq:stochDirichlet} \\
-\partial_\nu u(t,x,\randomvar) &= f(t,x,\randomvar) \tag {SN} \label{eq:stochNeumann} \\
\partial_\nu u(t,x,\randomvar) -\alpha({t},x,\randomvar) \partial_t u(t,x,\randomvar) 
&= f(t,x,\randomvar), \tag{SA} \label{eq:stochAcoustic}
\end{align}
where $\nu$ is the outer unit normal vector to $\Gamma$, $x \in \Gamma$. The variable $\randomvar$ ranges over the  stochastic space $\Xi$, with a probability measure $\probmeasure(\randomvar)$.

We recast the boundary value problems as time dependent boundary integral equations in a suitable scale of function spaces. We first consider the wave equation \eqref{eq:waveeq} with stochastic Dirichlet  \eqref{eq:stochDirichlet}, respectively Neumann conditions \eqref{eq:stochNeumann}, before acoustic boundary conditions \eqref{eq:stochAcoustic} are considered.
\subsection{Dirichlet and Neumann boundary conditions}\label{sec:dirneu}

For the Dirichlet problem  \eqref{eq:stochDirichlet}, we make an ansatz for the solution using the single layer potential in time domain,
\begin{equation}\label{singlay}
u(t,x,\randomvar) =\int_{\mathbb{R}^+} \int_{\Gamma} G(t- \tau,x,y)\ \phi(\tau,y,\randomvar)\ d\tau\ ds_y\ 
\end{equation}
almost everywhere for $x\in  \Gamma$ and $\probmeasure$-a.s.~for $\randomvar \in \Xi$.
Here, $G$ is a fundamental solution of the wave equation \eqref{eq:waveeq},
\begin{align}\label{eq:green}
 G(t-s,x,y)&= \frac{H(t-s-|x-y|)}{2\pi \sqrt{(t-s)^2+|x-y|^2}}\qquad &\text{for } d = 2,\\
 G(t-s,x,y)&=\frac{\delta(t-s-|x-y|)}{4\pi |x-y|} &\text{for } d=3,
\end{align}
involving the Heaviside function $H$ for $d=2$ and the Dirac distribution $\delta$ for $d=3$. 
The density $\phi$ is assumed to be causal: $\phi(\tau,y,\randomvar)=0$ for $\tau<0$. 

Taking the trace of the ansatz \eqref{singlay} on $\Gamma$, we obtain the single layer operator $\mathcal{V}$
$$\mathcal{V} \phi(t,x,\randomvar)={  \int_{\mathbb{R}^+} \int_\Gamma G(t- \tau,x,y)\ \phi(\tau,y,\randomvar)\ d\tau\ ds_y\, . }$$

As $u|_\Gamma = \mathcal{V} \phi$,  we obtain an equivalent formulation of the wave equation \eqref{eq:waveeq} with stochastic Dirichlet boundary conditions \eqref{eq:stochDirichlet}, $u|_\Gamma = f$, as a time dependent integral equation 
\begin{align} \label{symmhyp}
\mathcal{V} \phi = f .
\end{align}
Given a solution  for the density $\phi$ in \eqref{symmhyp}, the solution to the wave equation is obtained using the single layer ansatz \eqref{singlay}. \\

An analogous formulation for  the Neumann problem uses a double layer potential ansatz for the solution $u$:
\begin{align}\label{doubleansatz}
u(t,x,\randomvar)&=\int_{\mathbb{R}^+}\int_{\Gamma}  \frac{\partial G}{\partial \nu_y}(t- \tau,x,y)\ \psi(\tau,y,\randomvar)\ d\tau\ ds_y,
\end{align}
where $\psi$ again is a causal function, $\psi(\tau,y,\randomvar) = 0$ for $\tau< 0$. The resulting integral formulation for the wave equation \eqref{eq:waveeq} with stochastic Neumann boundary conditions \eqref{eq:stochNeumann}, $-\partial_\nu u = f$, is the hypersingular equation
\begin{align}\label{hypersingeq}
 \mathcal{W} \psi = f\ .
\end{align}
Here, 
\begin{equation}\mathcal{W} \psi(t,x,\randomvar)= -\int_{\mathbb{R}^+} \int_\Gamma \frac{\partial^2  G}{\partial \nu_x \partial \nu_y}(t- \tau,x,y)\ \psi(\tau,y,\randomvar)\ ds_y\ d\tau  \ . \label{Woperator}\end{equation}
Given a solution  for the density $\psi$ in \eqref{hypersingeq}, the solution to the wave equation is obtained using the double layer ansatz \eqref{doubleansatz}.\\

The numerical experiments will quantify errors in space, time and stochastic variables using natural energy and $L^2$ norms. More generally,  
space--time anisotropic Sobolev spaces provide a convenient mathematical setting to study time dependent boundary integral equations like \eqref{symmhyp} and \eqref{hypersingeq}, which we review in Appendix A. The  Sobolev space
$L^2_w(\Xi; H_\sigma^{r}(\mathbb{R}^+, \widetilde{H}^{s}(\Gamma)))$ with \begin{align*}
\lVert u \rVert_{L^2_w(\Xi; H_\sigma^{r}(\mathbb{R}^+, \widetilde{H}^{s}(\Gamma)))} := \left(\int_{\Xi} \lVert u \rVert_{r,s,\Gamma}^{2} d\pi\right)^{1/2}
\end{align*}
is used to measure the error of the numerical solution in the stochastic variables in an $L^2$-norm, the spatial $L^2$-error of the $s$-th derivative, and the temporal $L^2$-error of the $r+s$-th derivative. \\

Using these spaces, we define the bilinear form bilinear form associated to the Dirichlet, respectively Neumann problem  by
\begin{align} \label{eq:stochbilinearform}
  B_{D}^s(\phi, \varphi) = \int_\Xi \int_{\mathbb{R}^+}\int_\Gamma \mathcal{V} \partial_t {\phi}(t,x,\randomvar)\ \varphi(t,x,\randomvar)\ ds_x \ d_\sigma t \ w(\randomvar)d\randomvar\ , \quad (\phi,\varphi \in L^2_w(\Xi; H_\sigma^{1}(\mathbb{R}^+, \widetilde{H}^{-\frac{1}{2}}(\Gamma)))),\\
  B_{N}^s(\psi, \chi) = \int_\Xi \int_{\mathbb{R}^+}\int_\Gamma \mathcal{W} \partial_t {\psi}(t,x,\randomvar)\ \chi(t,x,\randomvar)\ ds_x \ d_\sigma t \ w(\randomvar)d\randomvar\ , \quad (\psi, \chi \in L^2_w(\Xi; H_\sigma^{1}(\mathbb{R}^+, \widetilde{H}^{\frac{1}{2}}(\Gamma)))).
  \end{align}
 Here $d_\sigma t = e^{-2\sigma t} dt$ with $\sigma>0$.

The  weak formulation of the  single-layer equation \eqref{symmhyp} then reads, given data $f \in L_w^2(\Xi;L^2_\sigma(\mathbb{R}^+,\widetilde{H}^{\frac{1}{2}}(\Gamma)))$:\\

\noindent Find $\phi \in L_w^2(\Xi;H^{1}_\sigma(\mathbb{R}^+,\widetilde{H}^{-\frac{1}{2}}(\Gamma)))$ such that for all $\varphi \in L_w^2(\Xi;H^{1}_\sigma(\mathbb{R}^+,\widetilde{H}^{-\frac{1}{2}}(\Gamma)))$
\begin{equation}\label{DP} B_{D}^s(\phi, \varphi) = (\dot{f},\varphi)\ .\end{equation}

Similarly, the weak formulation of the hypersingular equation \eqref{hypersingeq} reads as follows, for data $f \in L_w^2(\Xi;L^2_\sigma(\mathbb{R}^+,\widetilde{H}^{-\frac{1}{2}}(\Gamma)))$:\\

\noindent Find $\psi \in L_w^2(\Xi;H^{1}_\sigma(\mathbb{R}^+,\widetilde{H}^{\frac{1}{2}}(\Gamma)))$ such that for all $\chi \in L_w^2(\Xi;H^{1}_\sigma(\mathbb{R}^+,\widetilde{H}^{\frac{1}{2}}(\Gamma)))$
\begin{equation}\label{NP} B_{N}^s(\psi, \chi) = (\dot{f},\chi)\ .\end{equation}

Below we will consider Galerkin discretizations of \eqref{DP} and \eqref{NP} using piecewise polynomial, conforming approximations in the subspaces  $$V \subset L_w^2(\Xi;H^{1}_\sigma(\mathbb{R}^+,\widetilde{H}^{-\frac{1}{2}}(\Gamma))), \text{ respectively } W \subset L_w^2(\Xi;H^{1}_\sigma(\mathbb{R}^+,\widetilde{H}^{\frac{1}{2}}(\Gamma))) \ .$$ 
The discretization is discussed in detail in Section 3.

The resulting Galerkin boundary element method of the single-layer equation \eqref{DP} is given by:

\noindent Find $\phi_{h,\Delta t,\xi} \in V$ such that for all $\varphi_{h,\Delta t,\xi} \in V$
\begin{equation}\label{DPh} B_{D}^s(\phi_{h,\Delta t,\xi}, \varphi_{h,\Delta t,\xi}) = (\dot{f},\varphi_{h,\Delta t,\xi})\ .\end{equation}

For the discretization of the hypersingular equation \eqref{NP} we similarly consider:

\noindent Find $\psi_{h,\Delta t, \xi} \in W$ such that for all $\chi_{h,\Delta t,\xi} \in W$
\begin{equation}\label{NPh} B_{N}^s(\psi_{h,\Delta t,\xi}, \chi_{h,\Delta t,\xi}) = (\dot{f},\chi_{h,\Delta t,\xi})\ .\end{equation}

The following results detail the theoretical numerical analysis of the Galerkin approximations \eqref{DPh}, resp.~\eqref{NPh}, including the stability and convergence of the numerical methods. Readers interested in the implementation and performance may refer to Sections 3 and 4.

For the analysis, we first note a stochastic extension of the well-known mapping properties known for
the layer potentials between Sobolev spaces related to the energy, see e.g.~\cite{hd}. A detailed exposition of the mathematical background of time domain integral equations and 
their discretizations is available in the monograph by Sayas \cite{sayas2016retarded}, 
including methods based on convolution quadrature. See \cite{costabel04, hd} for more concise introductions.

\begin{theorem}\label{mapthm}
The single layer and hypersingular operators are continuous for $r \in \mathbb{R}$:
\begin{align*}
  \mathcal{V}&: L^2_w(\Xi; H_\sigma^{r+1}(\mathbb{R}^+, \widetilde{H}^{-\frac{1}{2}}(\Gamma))) \to L^2_w(\Xi;H_\sigma^{r}(\mathbb{R}^+, H^{\frac{1}{2}}(\Gamma)))  \ ,\\ \mathcal{W} &: L^2_w(\Xi; H_\sigma^{r+1}(\mathbb{R}^+, \widetilde{H}^{\frac{1}{2}}(\Gamma))) \to L^2_w(\Xi; H_\sigma^{r}(\mathbb{R}^+, H^{-\frac{1}{2}}(\Gamma)))\ .
  \end{align*}
\end{theorem}
\begin{proof}
Because the integral operators $\mathcal{V}$ and $\mathcal{W}$ do not depend on the stochastic variable, the proof follows from the mapping properties without stochastic variables, see \cite{hd, setup}.
\end{proof}
Theorem \ref{mapthm} implies that the continuity of the bilinear form $B_{D}^s$ from \eqref{eq:stochbilinearform} on $L^2_w(\Xi; H_\sigma^{1}(\mathbb{R}^+, \widetilde{H}^{-\frac{1}{2}}(\Gamma)))$.  Similarly, the bilinear form $B_{N}^s$ is continuous on $L^2_w(\Xi; H_\sigma^{1}(\mathbb{R}^+, \widetilde{H}^{\frac{1}{2}}(\Gamma)))$.  

The following theorem details the continuity and (weak) coercivity properties of $B_N^s$ and $B_D^s$: 
\begin{theorem}\label{DPbounds} a) For all
  $\phi,\varphi \in L^2_w(\randomvar;H^1_\sigma( \mathbb{R}^+, \widetilde{H}^{-\frac{1}{2}}(\Gamma)))$,
  $$|B_{D}^s(\phi,\varphi)| \lesssim \|\phi\|_{L^2_w(\Xi;H^1_\sigma( \mathbb{R}^+, \widetilde{H}^{-\frac{1}{2}}(\Gamma)))} \|\varphi\|_{L^2_w(\Xi;H^1_\sigma( \mathbb{R}^+, \widetilde{H}^{-\frac{1}{2}}(\Gamma)))}$$
  and
  $$\|\phi\|_{L^2_w(\Xi;H^0_\sigma( \mathbb{R}^+, \widetilde{H}^{-\frac{1}{2}}(\Gamma)))}^2 \lesssim B_{D}^s(\phi,\phi) . $$
b) For all
  $\psi,\chi \in L^2_w(\randomvar;H^1_\sigma( \mathbb{R}^+, \widetilde{H}^{\frac{1}{2}}(\Gamma)))$,
  $$|B_{N}^s(\psi,\chi)| \lesssim \|\psi\|_{L^2_w(\Xi;H^1_\sigma( \mathbb{R}^+, \widetilde{H}^{\frac{1}{2}}(\Gamma)))} \|\chi\|_{L^2_w(\Xi;H^1_\sigma( \mathbb{R}^+, \widetilde{H}^{\frac{1}{2}}(\Gamma)))}$$
  and
  $$\|\psi\|_{L^2_w(\Xi;H^0_\sigma( \mathbb{R}^+, \widetilde{H}^{\frac{1}{2}}(\Gamma)))}^2 \lesssim B_{N}^s(\psi,\psi) . $$
\end{theorem}
\begin{proof}
Again, the proof of Theorem \ref{DPbounds} follows immediately from the results without stochastic variables, see e.g.~\cite{hd} or in the specific notation used here, Proposition 3.1 of \cite{residual} for a) and Theorem 2.1 b) of \cite{china} for b).
\end{proof}

As a first consequence, we obtain the stability of the proposed methods:
\begin{theorem}Let $\sigma>0$, $r \in \R$. \\ a) Let $f \in L^2_w(\Xi, H^{r+1}_{\sigma}(\mathbb{R}^+,H^{\frac{1}{2}}(\Gamma)))$. Then there exists a unique solution $\phi \in L^2_w(\Xi, H^r_{\sigma}(\mathbb{R}^+,\widetilde{H}^{-\frac{1}{2}}(\Gamma)))$  of \eqref{DP}  and a unique solution $\phi_{h,\Delta t,\xi} \in V$  of \eqref{DPh}, and 
\begin{equation}
\|\phi\|_{L_{w}^{2}(\Xi,H_{\sigma}^{r}(\mathbb{R}^{+}),\widetilde{H}^{-1/2}(\Gamma))}, \|\phi_{h,\Delta t,\xi}\|_{L_{w}^{2}(\Xi,H_{\sigma}^{r}(\mathbb{R}^{+}),\widetilde{H}^{-1/2}(\Gamma))} \lesssim_\sigma \|f\|_{L_{w}^{2}(\Xi,H_{\sigma}^{r+1}(\mathbb{R}^{+}),\widetilde{H}^{1/2}(\Gamma))}\ .
\end{equation}
b) Let $f \in L^2_w(\Xi, H^{r+1}_{\sigma}(\mathbb{R}^+,H^{-\frac{1}{2}}(\Gamma)))$. Then there exists a unique solution $\psi \in L^2_w(\Xi, H^{r+1}_{\sigma}(\mathbb{R}^+,\widetilde{H}^{\frac{1}{2}}(\Gamma)))$  of \eqref{NP} and a unique solution $\psi_{h,\Delta t,\xi} \in W$  of \eqref{NPh}, and
\begin{equation}
\|\psi\|_{L_{w}^{2}(\Xi,H_{\sigma}^{r}(\mathbb{R}^{+}),\widetilde{H}^{1/2}(\Gamma))}, \|\psi_{h,\Delta t,\xi}\|_{L_{w}^{2}(\Xi,H_{\sigma}^{r}(\mathbb{R}^{+}),\widetilde{H}^{1/2}(\Gamma))}\lesssim_\sigma \|f\|_{L_{w}^{2}(\Xi,H_{\sigma}^{r+1}(\mathbb{R}^{+}),\widetilde{H}^{-1/2}(\Gamma))} \ .
\end{equation}
\end{theorem}
\begin{proof} The proof of this theorem follows immediately from Theorem 6 of \cite{graded}. In particular, measurability follows from the continuous dependence of the solution operator. \end{proof}

From the continuity and (weak) coercivity of the bilinear forms, it is now standard \cite{hd,setup,china} to conclude the convergence of the Galerkin solutions in \eqref{DPh}, \eqref{NPh}.  Details are omitted here.


\subsection{Acoustic boundary conditions}
For the acoustic boundary problem \eqref{eq:stochAcoustic}, we assume that the coefficient $\alpha({t},x,\randomvar)$ in the acoustic boundary condition \eqref{eq:stochAcoustic} is bounded from above and below,
$$0<\underline{\alpha} \leq \alpha({t},x,\randomvar) \leq \overline{\alpha} <\infty,$$
almost everywhere for $t\geq0$, $x\in  \Gamma$, $\probmeasure$-a.s. $\randomvar \in \Xi$.

In addition to the single layer operator $\mathcal{V}$ and the hypersingular operator $\mathcal{W}$, 
for the boundary integral formulation of this problem we require the double layer operator $\mathcal{K}$ and the adjoint double layer operator $\mathcal{K}'$ for $x \in \Gamma$, $t>0$: 
\begin{align}
\mathcal{K}\varphi(t,x,\xi)&= \int_{\mathbb{R}^+} \int_\Gamma \frac{\partial  {G}}{\partial \nu_y}(t- \tau,x,y)\ \varphi(\tau,y,\xi)\ ds_y\ d\tau\ , \label{Koperator}\\
\mathcal{K}' \varphi(t,x,\xi)&=  \int_{\mathbb{R}^+} \int_\Gamma \frac{\partial  {G}}{\partial \nu_x}(t- \tau,x,y)\ \varphi(\tau,y,\xi)\ ds_y\ d\tau\, . \label{Kpoperator}
\end{align}
{For the derivation and the analysis of the relevant boundary integral formulation, following \cite{hd,ha2003galerkin} we introduce the auxiliary \emph{interior} problem
\begin{align}\label{eq:waveeqint}
\partial_t^2 u^i - \Delta u^i = 0\ , \quad u=0\ \text{ for $x \in \mathcal{X}$ and }\ t\leq 0\ 
\end{align}
with the boundary condition
\begin{align}
\partial_\nu u^i(t,x,\randomvar) +\alpha(t,x,\randomvar) \partial_t u^i(t,x,\randomvar) 
= g(t,x,\randomvar), \tag{SA${}^i$} \label{eq:stochAcousticInt}
\end{align}
on $\Gamma$.
Denote by $U$ the function in $\mathbb{R}^d \setminus \Gamma$ which is equal to the solution $u$ in $\mathbb{R}^d \setminus \overline{\mathcal{X}}$, resp.~to the solution $u^i$ of the auxiliary problem in $\mathcal{X}$, Then the following representation formula holds
\begin{equation}\label{repformula}
U(t,x,\randomvar) =\int_{\mathbb{R}^+} \int_{\Gamma} G(t- \tau,x,y)\ p(\tau,y,\randomvar)\ d\tau\ ds_y + \int_{\mathbb{R}^+} \int_{\Gamma} \frac{\partial  {G}}{\partial \nu}(t- \tau,x,y)\ \varphi(\tau,y,\randomvar)\ d\tau\ ds_y\ , 
\end{equation}
where $\varphi = u^i - u$ and $p = \frac{\partial  u^i}{\partial \nu}-\frac{\partial  u}{\partial \nu}$ denote the jumps of the Dirichlet, resp.~Neumann traces on $\Gamma$. Letting $x \in \mathbb{R}^n\setminus \Gamma$ tend to $\Gamma$, one obtains the following classical relations on $\Gamma$:
\begin{align*}
&u^i = \mathcal{V}p + (\textstyle{\frac{1}{2}}-\mathcal{K})\varphi,   &u = \mathcal{V}p + (-\textstyle{\frac{1}{2}}-\mathcal{K}),\\
&\textstyle{\frac{\partial  u^i}{\partial \nu}} = (\textstyle{\frac{1}{2}}+\mathcal{K}')p-\mathcal{W}\varphi, &\textstyle{\frac{\partial  u}{\partial \nu}} = (-\textstyle{\frac{1}{2}}+\mathcal{K}')p-\mathcal{W}\varphi.
\end{align*}
We may substitute these into the acoustic boundary conditions \eqref{eq:stochAcoustic}, \eqref{eq:stochAcousticInt} to obtain a system of boundary integral equations on $\Gamma$:
\begin{align*}
(-\textstyle{\frac{1}{2}}+\mathcal{K}')p-\mathcal{W}\varphi-\alpha\partial_t(\mathcal{V}p-(\textstyle{\frac{1}{2}}+\mathcal{K})\varphi) &= f,\\
(\textstyle{\frac{1}{2}}+\mathcal{K}')p-\mathcal{W}\varphi+\alpha\partial_t(\mathcal{V}p+(\textstyle{\frac{1}{2}}-\mathcal{K})\varphi) &= g.
\end{align*}
Specifically, for scattering problems with an incoming wave $u_{inc}$ one uses $f = -\partial_\nu u_{inc}+\alpha \partial_t u_{inc}$, $g=-\partial_\nu u_{inc}-\alpha\partial_t u_{inc}$. This leads to the system
\begin{align*}
2(\mathcal{K}'p-\mathcal{W}\varphi)+\alpha\partial_t\varphi &= -\partial_\nu u_{inc},\\
p+2\alpha(\mathcal{V} \partial_t p-\mathcal{K}\partial_t \varphi) &=-2 \alpha \partial_t u_{inc}.
\end{align*}
}
For the time interval $[0,T]$, where $T$ is finite or $T=\infty$, we introduce the relevant bilinear form $a_T^s$ on $L^2_w(\Xi;H^1([0,T],\widetilde{H}^{\frac{1}{2}}(\Gamma))) \times L^2_w(\Xi;H^1([0,T],L^2(\Gamma)))$, given by
\begin{equation}\label{eq:bilinear_recall}
 a_T^s((\varphi,p),(\psi,q))=\int_\Xi\int_0^T \int_\Gamma \left( \alpha \dot{\varphi} \dot{\psi}  +  \frac{1}{\alpha} p q +  {2} \mathcal{K}' p \dot{\psi}
- {2} \mathcal{W} \varphi \dot{\psi} + {2} \mathcal{V} \dot{p} q {-2} \mathcal{K} \dot{\varphi} q \right) ds_x\, dt \ w(\randomvar) d\randomvar
\end{equation}
for $(\varphi,p),(\psi,q) \in L^2_w(\Xi;H^1([0,T],\widetilde{H}^{\frac{1}{2}}(\Gamma))) \times L^2_w(\Xi;H^1([0,T],L^2(\Gamma)))$.\\

The bilinear form now allows us to  reduce the wave equation \eqref{eq:waveeq}
  with acoustic boundary conditions \eqref{eq:stochAcoustic} to a boundary integral equation on $\Gamma$. The weak form is given by\\
\noindent Find $(\varphi,p) \in L^2_w(\Xi;H^1([0,T],\widetilde{H}^{\frac{1}{2}}(\Gamma))) \times L^2_w(\Xi;H^1([0,T],L^2(\Gamma)))$ such that
\begin{equation}\label{eq:acoustic_recall}
a_T^s((\varphi,p), (\psi,q)) = l(\psi, q)
\end{equation}
for all $(\psi,q)\in L^2_w(\Xi;H^1([0,T],{H}^{\frac{1}{2}}(\Gamma))) \times L^2_w(\Xi;H^1([0,T],L^2(\Gamma)))$.\\

Here,
\begin{equation}\label{eq:rechte_recall}
 l(\psi, q)=\int_\Xi \int_0^T \int_\Gamma F \dot{\psi} \,ds_x\, dt\ w(\randomvar) d\randomvar+\int_\Xi\int_0^T \int_\Gamma G q 
\, ds_x\, dt\, w(\randomvar) d\randomvar,
\end{equation}
where $F = -2\partial_\nu u_{inc}$, $G=-2 \partial_t u_{inc}$.\\

In the deterministic case, this formulation has been considered in \cite{hd,ha2003galerkin,residual}.\\

The Galerkin discretization of \eqref{eq:acoustic_recall} using piecewise polynomial, conforming approximations in the subspace $$Z \subset  L^2_w(\Xi;H^1([0,T],\widetilde{H}^{\frac{1}{2}}(\Gamma))) \times L^2_w(\Xi;H^1([0,T],L^2(\Gamma)))$$ reads:\\

\noindent Find $(\varphi_{h,\Delta t,\xi}, p_{h,\Delta t,\xi}) \in Z$ such that for all $(\psi_{h,\Delta t,\xi},q_{h,\Delta t,\xi})\in Z$
\begin{equation}\label{eq:acoustic_recallh}
a_T^s((\varphi_{h,\Delta t,\xi},p_{h,\Delta t,\xi}), (\psi_{h,\Delta t,\xi},q_{h,\Delta t,\xi})) = l(\psi_{h,\Delta t,\xi}, q_{h,\Delta t,\xi})\ .
\end{equation}

The discretization is discussed in detail in Section 3.\\

The following results detail the theoretical numerical analysis of the Galerkin approximations \eqref{eq:acoustic_recallh}, including the stability and convergence of the numerical methods. Readers interested in the implementation and performance may again refer to Sections 3 and 4.

We first note the following mapping properties of $\mathcal{K}$ and $\mathcal{K}'$:
\begin{theorem}\label{mapping2}
The following operators are continuous for $r \in \mathbb{R}$:
\begin{align*}
 \mathcal{K}&:  L^2_w(\Xi;H_\sigma^{r+1}(\mathbb{R}^+, \widetilde{H}^{\frac{1}{2}}(\Gamma))) \to  L^2_w(\Xi;H_\sigma^{r}(\mathbb{R}^+, H^{\frac{1}{2}}(\Gamma)))\ ,\\
\mathcal{K}' &:  L^2_w(\Xi;H_\sigma^{r+1}(\mathbb{R}^+, \widetilde{H}^{-\frac{1}{2}}(\Gamma))) \to  L^2_w(\Xi;H_\sigma^{r}(\mathbb{R}^+, H^{-\frac{1}{2}}(\Gamma)))\ .
\end{align*}
\end{theorem}
\begin{proof}
Because the integral operators $\mathcal{V}$ and $\mathcal{W}$ do not depend on the stochastic variable, the proof follows from the mapping properties without stochastic variables, see \cite{hd, setup}.
\end{proof}

The following theorem summarizes the continuity and coercivity properties of $a_T^s$. As above, the continuity follows from the mapping properties in Theorem \ref{mapthm} and Theorem \ref{mapping2}.
\begin{theorem}\label{APbounds} For all
  $(\varphi,p),(\psi,q) \in L^2_w(\Xi;H^1([0,T],\widetilde{H}^{\frac{1}{2}}(\Gamma))) \times L^2_w(\Xi;H^1([0,T],L^2(\Gamma)))$,
\begin{align*}
&a_T^s((\varphi,p),(\psi,q))\\& \lesssim (\|\varphi\|_{L^2_w(\Xi;H^1([0,T],\widetilde{H}^{\frac{1}{2}}(\Gamma)))}+ \|p\|_{L^2_w(\Xi;H^1([0,T],L^2(\Gamma)))})(\|\psi\|_{L^2_w(\Xi;H^1([0,T],\widetilde{H}^{\frac{1}{2}}(\Gamma)))}+ \|q\|_{L^2_w(\Xi;H^1([0,T],L^2(\Gamma)))})\ .
\end{align*}
  and
$$
\|p\|^2_{L_{w}^{2}(\Xi,H^{0}([0,T]),L^{2}(\Gamma))}+\|\dot{\varphi}\|_{L_{w}^{2}(\Xi,H^{0}([0,T]),L^{2}(\Gamma))}^2 \lesssim a_{T}((\varphi,p),(\varphi,p))\ .
$$
\end{theorem}
\begin{proof}
The upper estimate follows from the mapping properties of the integral operators in  Theorem \ref{mapthm} and Theorem \ref{mapping2}. The coercivity follows from Equation (64) of \cite{hd},  
$$ a_{T}^{s}((\varphi,p),(\varphi,p)) = 2\int_\Xi E(T,\omega) d\mu_\omega+ \int_\Xi \int_0^T \int_\Gamma \left( \alpha (\partial_t{\varphi}) (\partial_t{\varphi})  + \frac{1}{\alpha} p^2 \right) ds_x\, dt\ d\mu_\omega \ ,$$
where $E(T,\omega)= \frac{1}{2}\int_{\mathbb{R}^d \setminus \overline{\Omega}} \left\{(\partial_t u)^2 + (\nabla u)^2\right\}  dx$ is the total energy at time $T$. {While only time-independent $\alpha$ were considered in \cite{hd}, the proof of these formulas holds verbatim also when $\alpha$ depends on time.}
\end{proof}

We obtain the stability of the proposed method:
\begin{proposition}
Let $F \in L^2_w(\Xi;H^2([0,T],H^{-\frac{1}{2}}(\Gamma)))$, $G \in L^2_w(\Xi;H^1([0,T],H^{0}(\Gamma)))$. Then the weak form \eqref{eq:acoustic_recall} of the acoustic problem  and its discretization \eqref{eq:acoustic_recallh} admit unique solutions $(\varphi,p)\in  L^2_w(\Xi;H^1([0,T],\widetilde{H}^{\frac{1}{2}}(\Gamma)))\times L^2_w(\Xi;H^{1}([0,T], L^2(\Gamma)))$, resp.~$(\varphi_{h,\Delta t,\xi},p_{h,\Delta t,\xi})\in Z$, which depend continuously on the boundary conditions.
\end{proposition}
{As in Subsection \ref{sec:dirneu}, the continuity and (weak) coercivity of the bilinear form  allow to conclude the convergence of the Galerkin solutions in \eqref{eq:acoustic_recallh}}.  Details are omitted here.

\section{Stochastic space time boundary element method}\label{sec:SSTBEM}
This section establishes the space-time and stochastic discretization of \eqref{DP}, 
\eqref{NP} and \eqref{eq:acoustic_recall}. We
first consider the stochastic discretization which
uses the Stochastic Galerkin (SG) method.
The resulting SG system is the solved in space-time
by a Petrov-Galerkin method.
\subsection{Stochastic Galerkin discretization}
In order to derive the SG discretization
we expand the solution of
\eqref{DP}, \eqref{NP}, or \eqref{eq:acoustic_recall} 
into a generalized Fourier series
using  a suitable orthonormal basis. To do so, we assume that the  stochastic space $\Xi = \prod_{\stochdim=1}^\infty [a_\stochdim,b_\stochdim]$ is a product of intervals $[a_\stochdim,b_\stochdim]$, $\stochdim \in \N$ and that the probability measure on $\Xi$ is a product of probability measures on $[a_\stochdim,b_\stochdim]$, $\probmeasure(\randomvar)=\prod_{\stochdim=1}^\infty \probmeasure_\stochdim(\randomvar_\stochdim)$.
Let $\{\Psi_n^m\}_{n=0}^\infty: \Xi \to \R$ be an
orthonormal basis w.r.t. to the probability measure
$\probmeasure_\stochdim$, i.e. for all $i,j \in \N_0$ we have
\begin{align} \label{eq:orthonormal}
\int_{\R} \Psi_i^m(\randomvar) \Psi_j^m(\randomvar)  ~\mathrm{d}\probmeasure_\stochdim(\randomvar) = \delta_{i,j}.
\end{align}
Furthermore, we define the corresponding polynomial
approximation space for a  fixed polynomial degree of $J \in \N_0$ as follows, 
$\mathcal{P}_J(\Xi):= \{p: \Xi \to \R ~|~ p \text{ is a polynomial of degree } J\}$.
We let 
$\N_0^{\N}:= \{
\kappa= (\kappa_\stochdim)_{\stochdim \in \N}, \kappa_\stochdim \in \N_0, \text{ for all } \stochdim \in \N \}$,
and for $J\in \N_0$ we define the truncated multi-index set
\begin{align}\label{def:truncatedIndexSet}
\mathcal{\SGOrder}:= \{
\kappa= (\kappa_\stochdim)_{\stochdim \in \N}, ~\kappa_\stochdim\leq J , \text{ for all } \stochdim \in \N \}.
\end{align} 
By
\begin{align} \label{def:orthonormalPolynomial}
\Psi_\kappa(\randomvar) := \prod_{m=0}^n \Psi_{\kappa_m}^m(\randomvar_m)\quad \text{for all } \kappa \in \N_0^{\N}
\end{align} 
we denote the corresponding multivariate orthonormal polynomials 
and we define the tensor-polynomial approximation space of polynomials of degree $\SGOrder \in \N_0$,
$$
\polyapproxspace:= \bigotimes\limits_{\kappa \in \mathcal{\SGOrder}} 
\mathcal{P}_{\kappa_\stochdim}(\Xi) .
$$
Following \cite{ErnstUllmann2012,XiuKarniadakis2002}, 
the solution $\phi$ of \eqref{DP} and \eqref{NP} can be written as
\begin{align} \label{eq:fourierSeries}
\sol(t,x,\randomvar)= \sum \limits_{\kappa \in \N_0^{\N}} \phi^\kappa(t,x)\Psi_\kappa(\randomvar).
\end{align}
Moreover, we truncate the infinite series in \eqref{eq:fourierSeries} 
at polynomial degree $\SGOrder \in \N_0$  and write,
using the index set \eqref{def:truncatedIndexSet}, 
\begin{align} \label{eq:ansatz}
  \sol(t,x,\randomvar)= \sum \limits_{\kappa \in \mathcal{\SGOrder}} \phi^\kappa(t,x)\Psi_\kappa(\randomvar).
  \end{align}
The deterministic Fourier modes
$ \phi^\kappa= \phi^\kappa(t,x)$ in \eqref{eq:fourierSeries} are defined by
\begin{align} \label{def:FourierMode}
\phi^\kappa(t,x)= \E (\sol(t,x,\cdot)  \Psi_\kappa(\cdot))
 = \int \limits_\Xi \sol(t,x,\randomvar)  \Psi_\kappa(\randomvar) ~\mathrm{d} \probmeasure(\randomvar)
\qquad\forall~ \kappa \in \N_0^n.
\end{align}
From the Fourier modes \eqref{def:FourierMode} we immediately extract expectation and variance 
of $\sol$ via
$$\E (\sol(t,x,\cdot))= \sol^0(t,x) \text{ and }
\text{Var}( \sol(t,x,\cdot))=\sum \limits_{\kappa \in \N_0^{\N}\backslash (0,0,\ldots)} \sol^{\kappa}(t,x)^2.$$

Using the ansatz \eqref{eq:ansatz}
and testing against
$\psi(t,x,\randomvar)_\kappa:=\psi(t,x)\Psi_\kappa(\randomvar)$
yields, in combination with the
the orthonormality relation \eqref{eq:orthonormal},
the following SG system 
of \eqref{DP}
\begin{align} \label{eq:SGsystem}
B_{D}^s(\sol^\kappa(t,x,\randomvar),\psi(t,x,\randomvar)_{\kappa}) = 
(\partial_t f(t,x,\randomvar) , \psi(t,x,\randomvar)_{\kappa})
\end{align}
for all $\kappa \in \mathcal{\SGOrder}$.

{Using the same ansatz and test functions in stochastics yield the following SG System for \eqref{eq:acoustic_recall}:
 \begin{equation}\label{eq:acoustic_stochdisc}
a_T^s((\varphi^{\kappa},p^{\kappa}), (\psi \Psi_{j},q \Psi_{j})) = l(\psi \Psi_{j}, q \Psi_{j}).
\end{equation}}
{For our computations in Section \ref{sec:numericalExpetiments}, we choose $\Xi=[-1,1]^{n+1}$ with Legendre polynomials of degree $\SGOrder$ as basis functions.}

\subsection{Space-time discretization of the Stochastic Galerkin system of the stochastic Dirichlet problem}\label{sec:stdisc_of_stochdir}
{ For the discretization in a space-time setting, we use a tensor product ansatz, dividing the space from the time components. In space, we mesh the hypersurface $\Gamma$ into triangular faces $\Gamma_{i}$ with diameter $h_{i}$ and $h=\max_{i} h_{i}$. Here, we introduce the space of piecewise polynomial functions $V_{h}^{p}$ with degree $p\geq 0$. Further the corresponding basis $\{ \varphi_{j}^{p} \}$ is continuous fpr $p \geq 1$. In time, we decompose $\mathbb{R}^{+}$ into subinervals $I_{n}=[t_{n-1},t_{n})$ with time step $(\Delta t)_{n}=| I_{n} |$ for $n=1,\ldots , $ and $(\Delta t)= \sup_{n} (\Delta t)_{n}$. Here, we denote with $\{\beta_{j}^{q} \}$ a corresponding basis of the space $V_{(\Delta t)}^{q}$ of piecewise polynomial functions of degree $q \geq 0$. The basis functions are continuous and vanish at $t=0$ for $q\geq 1$. Therefore, we divide the space-time cylinder $\mathbb{R}^{+} \times \Gamma = \bigcup_{j,k}[t_{j-1}, t_j)\times \Gamma_k$ by local tensor products $V_{h,\Delta t}^{p,q}=V_{h}^{p} \otimes V_{(\Delta t)}^{q}$ with the basis functions $\{ \beta_{j}^{q} \varphi_{k}^{p} \}$. Altogether the $\kappa-$th (space-time dependent) stochastic mode in \eqref{eq:SGsystem} is given by $\phi_{h,\Delta t}^{\kappa} \in V_{h,\Delta t}^{p,q}$ for all $\kappa \in \SGOrder$. Finally the fully discrete numerical approximaion of \eqref{DP} becomes
\begin{align} \label{def:fullydiscrete}
\phi_{h,\Delta t}^\SGOrder(t,x,\randomvar):= \sum \limits_{\kappa \in \mathcal{\SGOrder}}
\phi^\kappa_{h, \Delta t}(t,x) \Psi_\kappa(\randomvar) \in  V:=
V^{p,q}_{h,\Delta t} \otimes \polyapproxspace .
\end{align} 
This leads to the following discrete variational formulation of \eqref{DPh}: Find $\phi_{h,\Delta t}^{\SGOrder} \in V$ such that for all $l \in \SGOrder$ and $\psi_{h,\Delta t}^{l} \in V_{h,\Delta t}^{p,q}$
 \begin{align} \label{eq:SGdisc}
B_{D}^s(\sol^\SGOrder_{h,\Delta t},\psi^l_{h,\Delta t} \Psi_l) = 
(\partial_t f , \psi^l_{h,\Delta t}\Psi_l \big) .
\end{align}
For our computations in Section \ref{sec:numericalExpetiments}, we set a uniform space mesh with diameter $h$ and an equidistant timemesh with length $\Delta t$. Furthermore we choose piecewise constant ansatz and test functions in space and time, i.e. $\phi_{h,\Delta t}^{\kappa}, \psi_{h,\Delta t}^{l} \in V_{h,\Delta t}^{0,0}$. This leads to a lower triangular space-time Toeplitz matrix system. As usual, we solve this space-time system by backsubstitution, leading to a marching-on-in-time (MOT) time stepping scheme \cite{terrasse}. However, we have to note the stochastic contribution. Benefitting from the orthonormal polynomial basis (see \eqref{eq:orthonormal}), we get a block diagonal matrix system in the stochastic space time framework (see \eqref{eq:appstochspactimsys}), where we apply the MOT scheme $(\SGOrder+1)^{n}$ times (see \eqref{eq:appmotdirstoch}). For more details about the computation, {see} the Appendix \ref{appendix:discstochdir}.
}
\subsection{Space-time discretization of the Stochastic Galerkin system of the stochastic acoustic problem}
{ The stochastic space-time discretization spaces remain the same as in the subsections before. Further we compuate the stochastic acoustic problem on a uniform mesh with equidistant $(\Delta t)$ and $\Xi=[-1,1]^{n+1}$ with Legendre polynomials of degree $\SGOrder$. {For simplicity, we detail the implementation for a time-independent absorption coefficient $\alpha$, which we take to be of the form} $\alpha(x,\xi)=\alpha_{1}(x)\alpha_{2}(\xi)$. 
In space-time, we choose the following ansatz and test functions:
}
\begin{itemize}
\item Ansatz function $\phi_{h,\Delta t}^{\kappa} \in V_{h,\Delta t}^{1,1}$ piecewise linear in space and time
\item Ansatz function $p_{h,\Delta t}^{\kappa} \in V_{h, \Delta t}^{0,1}$ piecewise constant in space and piecewise linear in time
\item Test function $\dot{\psi}_{h,\Delta t}^{l} \in V_{h,\Delta t}^{1,0}$ piecewise linear in space and piecewise constant in time
\item Test function $q_{h,\Delta t}^{l} \in V_{h,\Delta t}^{0,0}$ piecewise constant in space and time
\end{itemize}

The Galerkin discretization \eqref{eq:acoustic_recallh} contains the retarded potential operators, namely the single layer potential $\mathcal{V}$, the double layer potential $\mathcal{K}$, the adjoint double layer potential $\mathcal{K}'$ and the hypersingular integral operator $\mathcal{W}$. For the precise discretization of these retarded operators with this choice of ansatz and test functions in space and time, we refer to the Appendix in \cite{wave-wave-coupling}. This leads again to a space-time system for which we apply the MOT scheme. Unfortunately, since \eqref{eq:acoustic_recallh} depends on $\alpha$ as well as $\frac{1}{\alpha}$, we lose the stochastic block space-time diagonal system. At least we are able to separate the stochastic space from the time, leading to a lower triangle block stochastic space Toeplitz matrix, where we solve this system via an MOT scheme with stochastic space matrices. 

Precisely, for every time step {we solve:}
{\small{ 
\begin{align}\label{eq:StochasticAcousticDiscreteSystem}
  &\begin{pmatrix} \frac{\Delta t}{2} S(\frac{1}{\alpha_{2}}) \otimes I_{const}(\frac{1}{\alpha_{1}}) + 2 diag(V^{0}) & 2 diag(K^{0}) \\ 2 diag({K'}^{0}) & S(\alpha_{2}) \otimes I_{lin}(\alpha_{1}) - 2 diag({W}^{0})  \end{pmatrix}
	\begin{pmatrix} p^{o} \\ \varphi^{o} \end{pmatrix} = \\ &\hspace*{-0.4cm}\begin{pmatrix} G - \frac{\Delta t}{2} S(\frac{1}{\alpha_{2}}) \otimes I_{const}(\frac{1}{\alpha_{1}}) p^{o-1} -2 \sum_{m=1}^{o-1} diag(V^{o-m}) p^{m}- 2 \sum_{m=1}^{o-1} diag(K^{o-m}) \varphi^{m} \nonumber \\ 
	F + S(\alpha_{2}) \otimes I_{lin}(\alpha_{1}) \varphi^{o-1}  + 2 \sum_{m=1}^{o-1} diag(W^{o-m})  \varphi^{o} - 2 \sum_{m=1}^{o-1} diag({K'}^{o-m}) p^{m} \end{pmatrix}
\end{align} 
}}
with $\otimes$ the Kronecker-Product and $p^{o}$ resp. $\varphi^{o}$ contains stochastic entries $\begin{pmatrix} p_{0}^{o} \\ \vdots \\ p_{(\SGOrder+1)^{n}-1}^{o} \end{pmatrix}$ resp. $\begin{pmatrix} \varphi_{0}^{o} \\ \vdots \\ \varphi_{(\SGOrder+1)^{n}-1}^{o} \end{pmatrix}$, where each stochastic entry contains entries in space. Further
\begin{equation*}
S_{i,j}\left(\frac{1}{\alpha_{2}}\right) := \int_{\mathbb{R}} \frac{1}{\alpha_{2}(\xi)} \Psi_{i}(\xi) \Psi_{j}(\xi) w(\xi) d\xi ,
\end{equation*}
and $I_{const}$ resp. $I_{lin}$ the identity with constant ansatz and test functions in space resp. linear ansatz and test functions in space.
$V^{k},K^{k},{K'}^{k},W^{k}$ at timestep $k+1$, $k\in \mathbb{N}_{0}$ are space matrices for the corresponding single layer, double layer, adjoint double layer and hypersingular integral operators. For more detals on these retarded potential matrices, see \cite{wave-wave-coupling,review} and for more details on how to construct \eqref{eq:StochasticAcousticDiscreteSystem} {see} the Appendix \ref{appendix:discstochacoustic}.

\section{Numerical experiments}
\label{sec:numericalExpetiments}
\subsection{Numerical example for the single layer potential with stochastic Dirichlet data} 
\label{subsec:singlelayerStochastic}
In this numerical example we consider the stochastic Dirichlet Problem \eqref{DP} on the unit sphere 
$\Gamma = S^2$ (see \figref{fig:SphereMesh} for a 2nd and 4th refinement level) with time $[0,2]$ and $\Xi=[-1,1]^3$. The stochastic right hand side with $k=(0.2,0.2,0.2)$ is given via 
\begin{align*}
f(t,x,\randomvar) = \exp{\Big(\randomvar_1-\frac{1}{10t^2}\Big)} \Big( \cos(\randomvar_2 + \randomvar_3)\cos(|k|t-k\cdot x)  
+ \sin(\randomvar_2 + \randomvar_3)\sin(|k|t-k\cdot x)   \Big).
\end{align*}
The spatial and temporal, uniform refinements are prescribed
in \tabref{table:meshesV}. The CFL (Courant-Friedrichs-Levy) rate is fixed at $0.605$.

To quantify the error in the numerical approximation
we consider the energy norm given by 
\begin{align} \label{def:energyError}
E(\phi_{h,\Delta t}^\SGOrder ):= & \frac{1}{2} 
\int_\Xi \int_{\mathbb{R}^+}\int_\Gamma \mathcal{V}
 \Big(\partial_t \phi_{h,\Delta t}^\SGOrder )\Big) 
 \phi_{h,\Delta t}^\SGOrder\ ds_x \ d_\sigma t \ d \probmeasure(\randomvar)
\\ & - 
\int_\Xi \int_{\mathbb{R}^+}\int_\Gamma \partial_t
f(t,x,\randomvar) \phi_{h,\Delta t}^\SGOrder \ ds_x \ d_\sigma t \  d \probmeasure(\randomvar).  \notag
\end{align}
Upon using the orthonormality relation \eqref{eq:orthonormal}, the expression in
\eqref{def:energyError} can be simplified to
\begin{align} \label{def:simplerEnergyError}
 E(\phi_{h,\Delta t}^\SGOrder )= \sum \limits_{\kappa \in \mathcal{\SGOrder}}
  \Big(\frac{1}{2}\int_{\mathbb{R}^+}\int_\Gamma \mathcal{V} 
  \Big(\partial_t \phi_{h,\Delta t}^\kappa \Big)  \phi_{h,\Delta t}^\kappa \ ds_x \ d_\sigma t 
 - 
\int_{\mathbb{R}^+}\int_\Gamma \partial_t
f(t,x,\randomvar) \phi_{h,\Delta t}^\kappa \ ds_x \ d_\sigma t  \Big).
\end{align}
To obtain a corresponding benchmark energy we either consider
a sufficient fine stochastic resolution or perform an extrapolation of the energy with respect to the
space-time degrees of freedom (DOF). 

\begin{table}
\centering
\begin{tabular}{|l|l|l|} 
 \multicolumn{3}{c}{ }\\
 \cline{1-3}	
 Refinement level & no. spatial elements & no. time-steps  \\
 \hline
 1 & 80      & 5  \\
 \cline{1-3}
 2 & 320     & 10 \\
 \cline{1-3}
 3 & 1280    & 20 \\
 \cline{1-3}
 4 & 5120    & 40 \\ 
 \cline{1-3}
 5 & 20480   & 80 \\
 \hline
\end{tabular}
\caption{Space-time mesh hierarchy for the numerical
experiment from \secref{subsec:singlelayerStochastic}. 2nd and 4th refinement level depicted in \figref{fig:SphereMesh} }
\label{table:meshesV}
\end{table}

We are first interested
in the convergence of the numerical approximation
with respect to the number of orthonormal polynomials, i.e. an increasing SG polynomial degree,
for a fixed space-time mesh. In order to compute a benchmark energy we compute the numerical approximation 
for the corresponding refinement level with a SG polynomial degree of 12.
The mean and variance of the benchmark solution is depicted in \figref{fig:VMeanAndVarT3} for $T=2$.

In \figref{fig:VError}a) we plot the relative
error between the approximated energy and the
benchmark energy versus the SG polynomial degree in one stochastic dimension, where we fix the space-time mesh to be either the fourth or fifth refinement level, cf. \tabref{table:meshesV}.
We can see clearly that the numerical error decays
spectrally, which we expect from approximating
a smooth solution with orthonormal polynomials.
In \figref{fig:VError}b) we also 
plot the convergence of the approximated energy 
with respect to increasing space-time DOFs
(cf. \tabref{table:meshesV}).
In order to compute the benchmark energy, 
we fix the SG polynomial degree to fourth and five and 
perform an extrapolation of the energy with regard to the space-time DOF.
To estimate the rate of convergence we use a linear least squares
fit using all data points. We compute a slope of roughly one-half, which 
corresponds to a rate of one half in terms of the space-time DOFs
and in terms of the mesh width $h$ to a rate of $3/2$.
	
\begin{figure}
  \includegraphics[trim=180mm 10mm 180mm 70mm, clip,width=0.50 \linewidth]{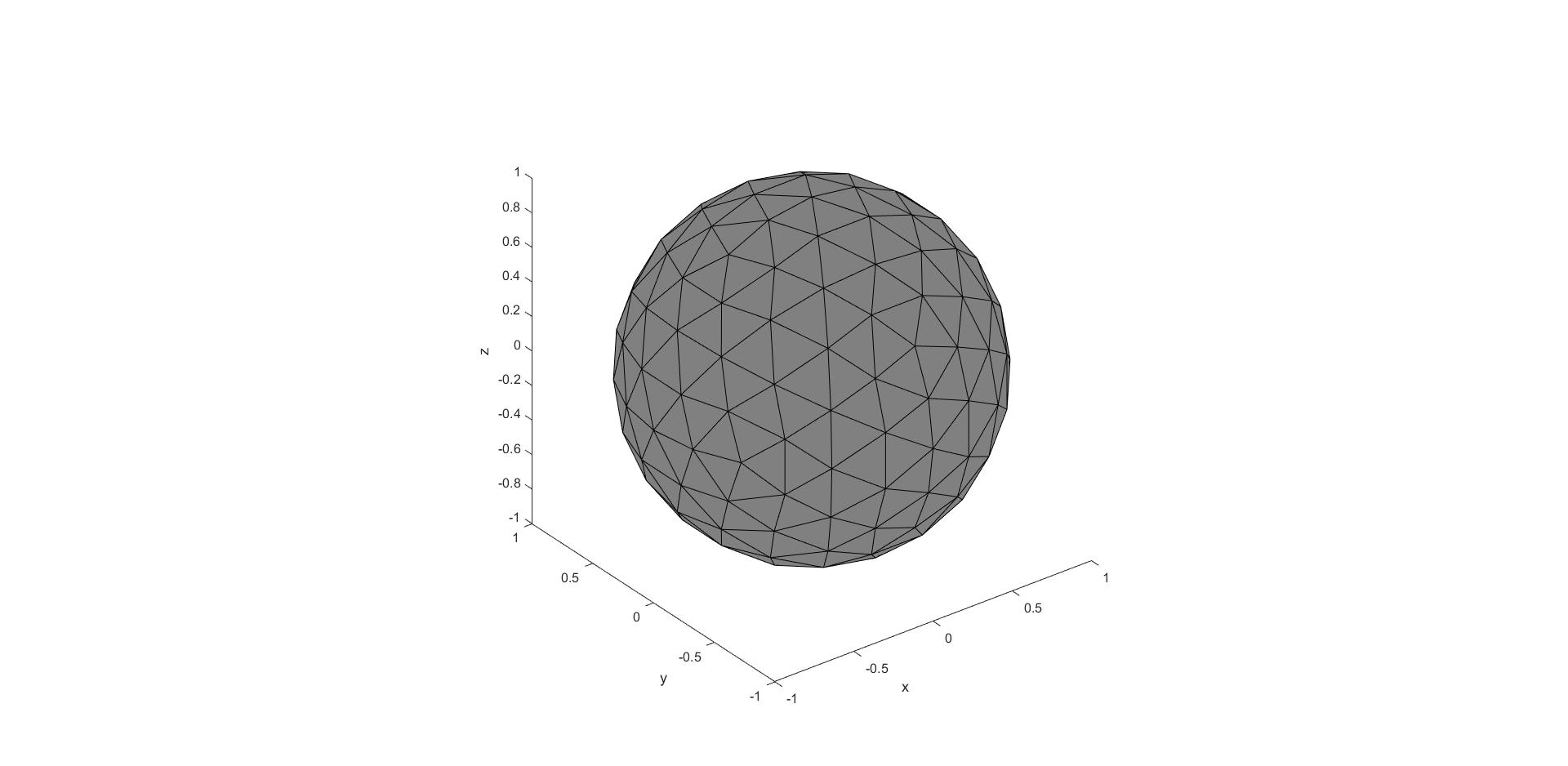}
  \includegraphics[trim=180mm 10mm  180mm 70mm, clip,width=0.50 \linewidth]{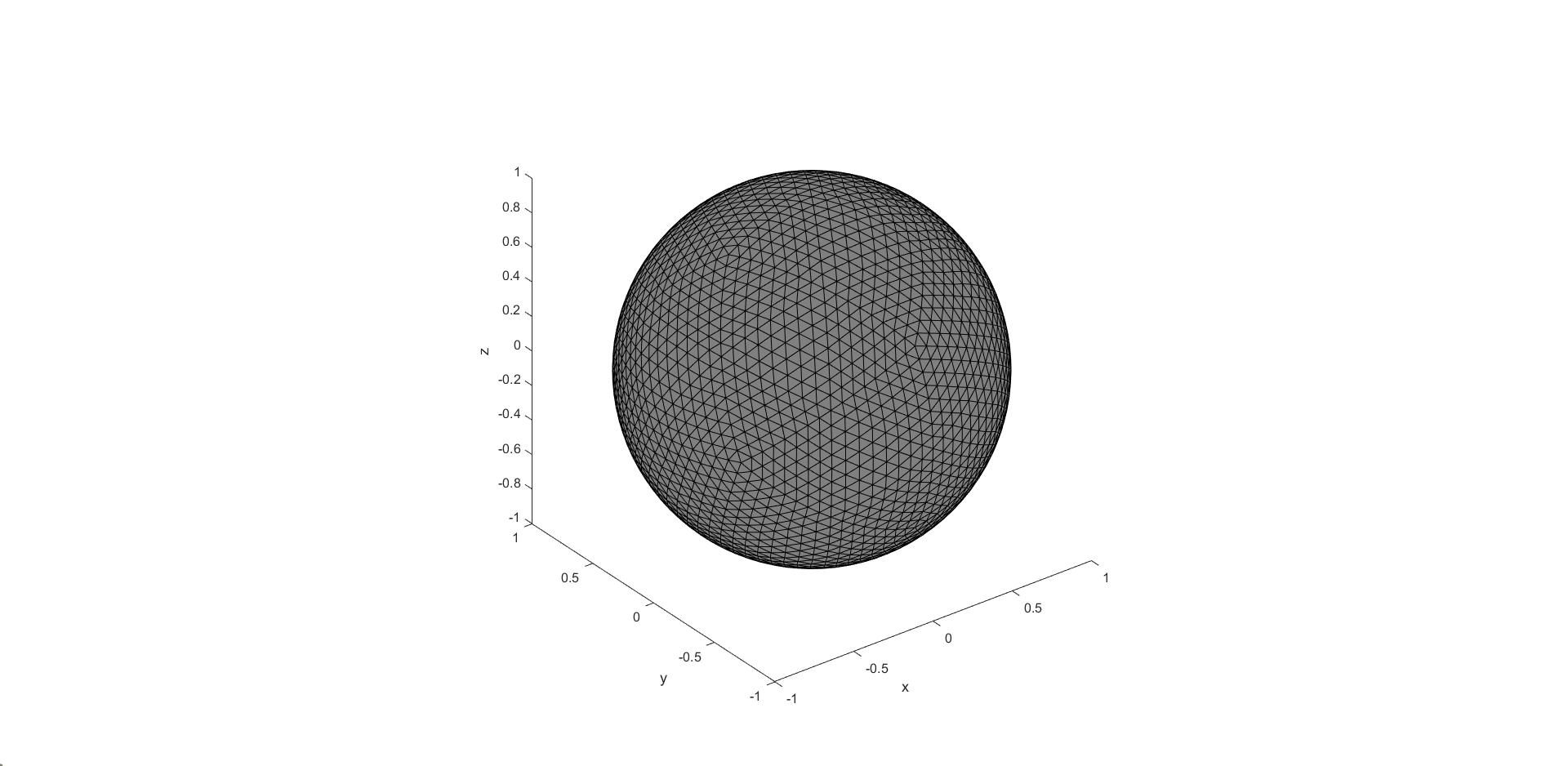}
  \caption{2nd and 4th spatial refinement level of the sphere 
  \secref{subsec:singlelayerStochastic}.}
  \label{fig:SphereMesh}
\end{figure}

\begin{figure}
  \includegraphics[trim=0mm 0mm 0mm 6mm, clip,width=0.55 \linewidth]{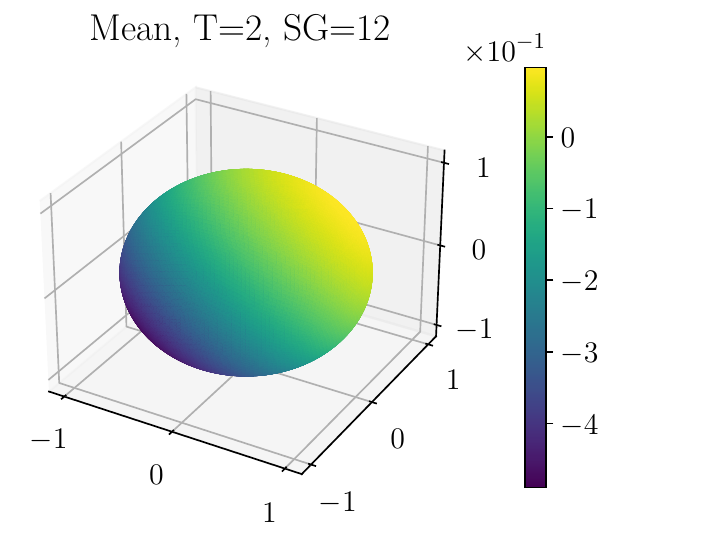}
  \includegraphics[trim=0mm 0mm 0mm 10mm, clip,width=0.55 \linewidth]{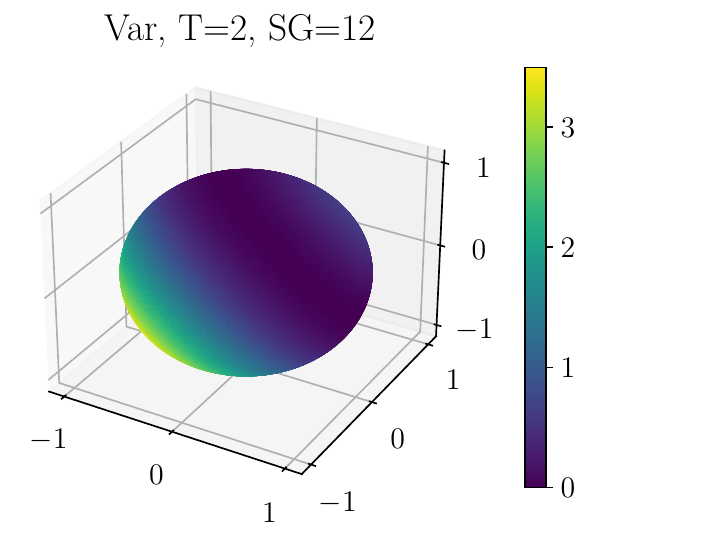}
  \caption{Plot of a) mean and b) variance of benchmark solution at time $T=2$, refinement level 5 and SG polynomial degree 12. 
  Numerical experiment from 
  \secref{subsec:singlelayerStochastic}.}
  \label{fig:VMeanAndVarT3}
\end{figure}

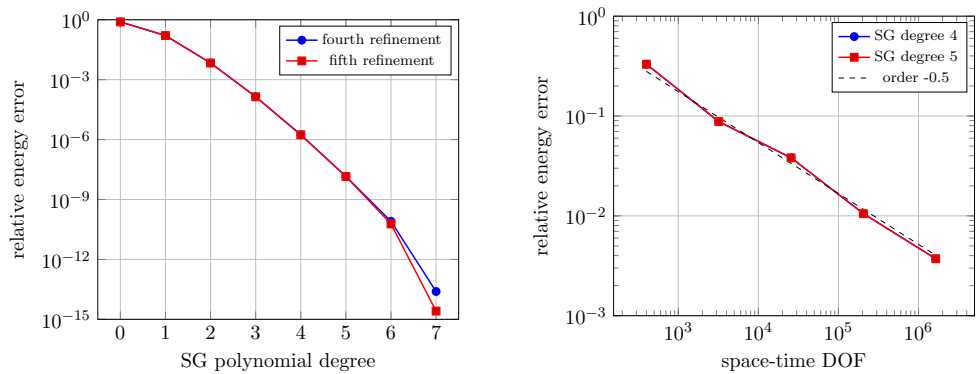
\begin{figure}
  \centering
\begin{tikzpicture}[scale=0.7]
\begin{axis}[
xtick={0,1,2,3,4,5,6,7},
xticklabels={0,1,2,3,4,5,6,7},
xmin = -0.5, xmax = 7.5,
ymin = 1e-15, ymax = 1,
ymode = log,
ylabel = {relative energy error},
xlabel = {SG polynomial degree},
xmode = linear,
grid = major,
cycle list name = color,
legend style={font=\scriptsize},
]
\addplot 
  table[row sep=crcr]{%
0			0.782952588418405 \\
1			0.161692794269475 \\
2			0.006857353957845 \\
3			0.000139428771247 \\
4			1.73997908822848E-06 \\
5			1.42946731827474E-08 \\
6			8.09494585428704E-11\\
7			2.47499255250684E-14 \\
};
\addlegendentry{fourth refinement}
\addplot
  table[row sep=crcr]{%
0			0.783495461867189 \\
1			0.161684187842095 \\
2			0.00685743828021 \\
3			0.000139424983689 \\
4			1.73983262217528E-06 \\
5			1.42160875067962E-08 \\ 
6			5.9732828396437E-11 \\
7			2.58656500525325E-15 \\
};
\addlegendentry{fifth refinement}
\end{axis}
\end{tikzpicture}
\qquad 
\begin{tikzpicture}[scale=0.7]
  \begin{axis}[
  xmin = 0, xmax = 5000000,
  ymin = 0.001, ymax = 1,
  ymode = log,
  ylabel = {relative energy error},
  xlabel = {space-time DOF},
  xmode = log,
  grid = major,
  cycle list name = color,
  legend style={font=\scriptsize},
  ]
  \addplot 
    table[row sep=crcr]{%
  400			0.330116440791064 \\
  3200			0.08796265658808 \\
  25600			0.038187588354382 \\
  204800			0.010533473160883 \\
  1638400			0.003724145150753 \\
  };
  \addlegendentry{SG degree 4}
  \addplot
    table[row sep=crcr]{%
  400			0.330116441772006 \\
  3200			0.087962656717812 \\
  25600			0.038187588357809 \\
  204800			0.010533473057372 \\
  1638400			0.003724145114157\\
  };
  \addlegendentry{SG degree 5}
  \addplot [color=black,dashed]
    table[row sep=crcr]{%
  400	2.80E-01 \\
  1638400	4.03E-03\\
  };
  \addlegendentry{order -0.5}
  \end{axis}
  \end{tikzpicture}
\caption{Plot of relative energy errors for a) a fixed space-time mesh and b) a fixed SG polynomial degree. Numerical experiment from 
\secref{subsec:singlelayerStochastic}.}
\label{fig:VError}
\end{figure}

\subsection{Numerical examples for an acoustic problem}
\label{subsec:Acoustic}
In this numerical example we solve
the stochastic acoustic problem
\eqref{eq:rechte_recall} on the unit cube $\Gamma =[-1,1]^3$ (see \figref{fig:CubeMesh} for a 2nd and 4th refinement level), time $[0,3]$ and $\Xi=[-1,1]$. The spatial and temporal, uniform refinements are shown in \tabref{table:meshesAcoustic}, where we hold the CFL rate at $0.387$. The right-hand side is deterministic and given by

\begin{align*}
  &G(t,x)=F(t,x) = \Big(\big(2 \pi (t\lVert x\rVert-\lVert x\rVert^2) \sin(\tfrac{\pi(\lVert x\rVert-t)}{0.9})+0.9 t (1+\cos(\tfrac{\pi(\lVert x\rVert-t)}{0.9}))\big) \frac{(1+\cos(\tfrac{\pi(\lVert x\rVert-t)}{0.9}))}{1.8 \lVert x\rVert^3 } x {\vec{n}}_{x} \\ &- \frac{1}{1.8 \lVert x\rVert} \big( (1+\cos(\tfrac{\pi(\lVert x\rVert-t)}{0.9})) (2 \pi (t-\lVert x\rVert)) \sin(\tfrac{-\pi(\lVert x\rVert-t)}{0.9}) -0.9  (1+\cos(\tfrac{\pi(\lVert x\rVert-t)}{0.9})) \big) \Big) H(0.9-|\lVert x\rVert - t|) .
\end{align*}
Here $\lVert x \rVert$ is the euclidean norm of $x$, ${\vec{n}}_{x}$ the exterior normal vector of the face $\Gamma_{i}$, where $x$ lies and $H$ the Heaviside function. 


Since we are using different ansatz and test functions we are not able
to compute the energy of the solution, but rather quantify the error in the 
$L^2((0,T)\times\Gamma,L^2(\Xi))$-norm,
i.e. we consider the difference 
$$\big | \|u_h\|_{L^2((0,T)\times\Gamma,L^2(\Xi))}- \|u_{\text{ref}}\|_{L^2((0,T)\times\Gamma ,L^2(\Xi))}\big| . $$
Similarly to the stochastic Dirichlet problem we compute the benchmark 
norm either by considering a sufficiently fine 
stochastic resolution or we perform an extrapolation of the norm with respect to the space-time DOFs.
The corresponding space-time refinement levels are detailed in \tabref{table:meshesAcoustic}.
\begin{table}
\centering
\begin{tabular}{|l|l|l|} 
 \multicolumn{3}{c}{ }\\
 \cline{1-3}	
 Refinement level & no. spatial elements & no. time-steps  \\
 \hline
 1 & 74      & 3  \\
 \cline{1-3}
 2 & 290     & 6 \\
 \cline{1-3}
 3 & 1154    & 12 \\
 \cline{1-3}
 4 & 4610    & 24 \\ 
 \cline{1-3}
 5 & 18434   & 48 \\
 \hline
\end{tabular}
\caption{Space-time mesh hierarchy for the numerical
experiment from \secref{subsec:Acoustic}. The 2nd and 4th spatial refinement level are displayed in \figref{fig:CubeMesh}. }
\label{table:meshesAcoustic}
\end{table}

\begin{figure}
  \hspace*{1.5cm}\includegraphics[trim=180mm 10mm 180mm 20mm, clip,width=0.40 \linewidth]{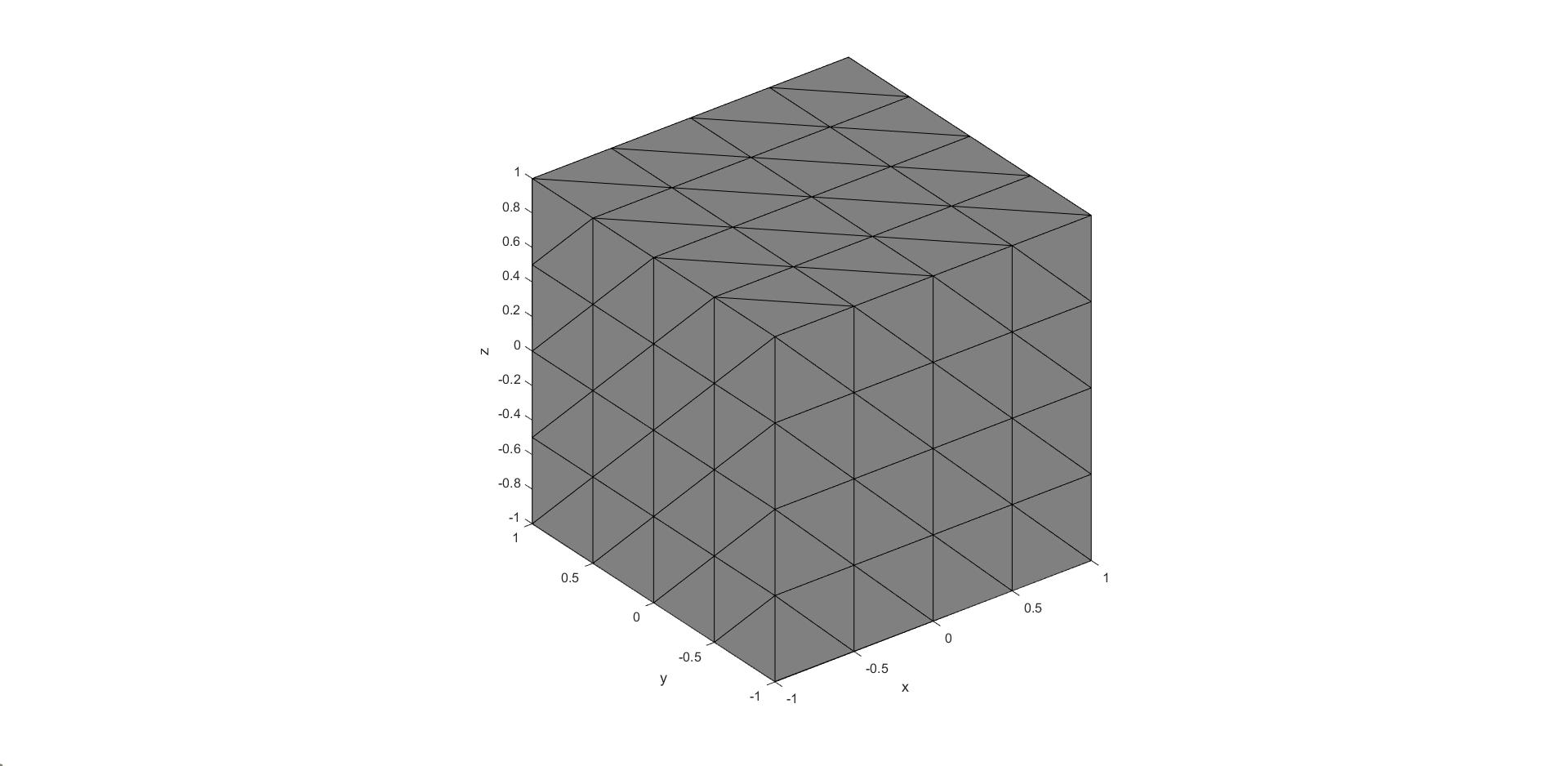}
  \includegraphics[trim=180mm 10mm  180mm 20mm, clip,width=0.40 \linewidth]{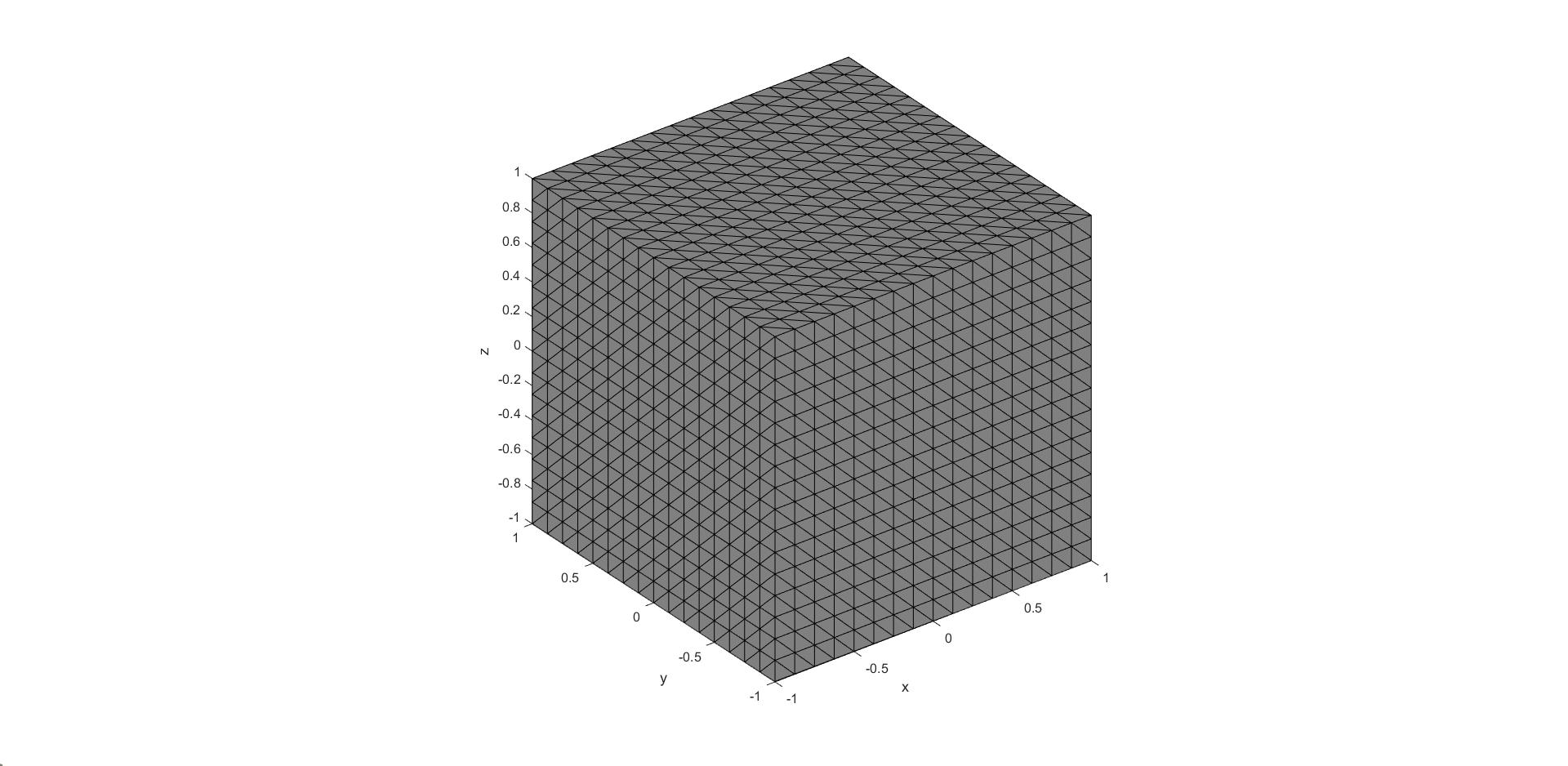}
  \caption{2nd and 4th spatial refinement level of the unit cube 
  \secref{subsec:Acoustic}.}
  \label{fig:CubeMesh}
\end{figure}

Similarly to the previous experiment for the stochastic Dirichlet problem 
we consider the convergence of the relative error with respect to the number 
of orthonormal polynomials, i.e the SG polynomial degree, for a fixed 
space-time mesh or with increasing refinement levels (see \tabref{table:meshesAcoustic}) for a fixed 
SG polynomial degree.

For the following numerical example we only change the stochastic parts of the impedance function $\alpha$. 
Our first numerical examples considers a stochastic coefficient $\alpha \sim \mathcal{U}(0.1,1.9)$. 

In \figref{fig:UniformAcousticMeanAdVarianceT1} we present the mean and variance of $p$ at time $T=1$ for the highest refinement level of the cube, c.f. \tabref{table:meshesAcoustic} for SG polynomial degree 4. 

{At the edges of the cube, where the mean has the lowest value, the variance seems to have the highest value. This 
indicates, that the stochastic coefficients influences the strength of the edge singularities.}


Further \figref{fig:UniformAcousticError}a) shows the convergence of $(p,\varphi)$ each for the relative $L^2((0,T)\times\Gamma;L^{2}(\Xi))$-error with increasing 
SG polynomial degree for a fixed space-time mesh. 

{With the increase of each SG polynomial degree, we almost see an error reduction of $10^{-2}$ for $\varphi$ until SG polynomial degree 4. Afterwards a clear reduction ins't visible anymore, whereas for $p$ the convergence rate reduces with the increase of each SG polynomial degree until $6$. Since ansatz and test functions aren't conform, this leads to a more unknown behaviour.}

\figref{fig:UniformAcousticError}b) depicts the convergence when increasing the 
space-time refinement level. Here we consider three different fixed SG polynomial degrees.
Computing the rate of convergence in terms of DOFs reveals a rate of $2/3$  for both components, 
the solution on $\Gamma$, $\varphi$ and it's normal derivative on $\Gamma$, $p$. This corresponds to a rate of $2$ in terms of the mesh width $h$.
\begin{figure}
  \includegraphics[trim=0mm 0mm 0mm 6mm, clip,width =0.5\linewidth]{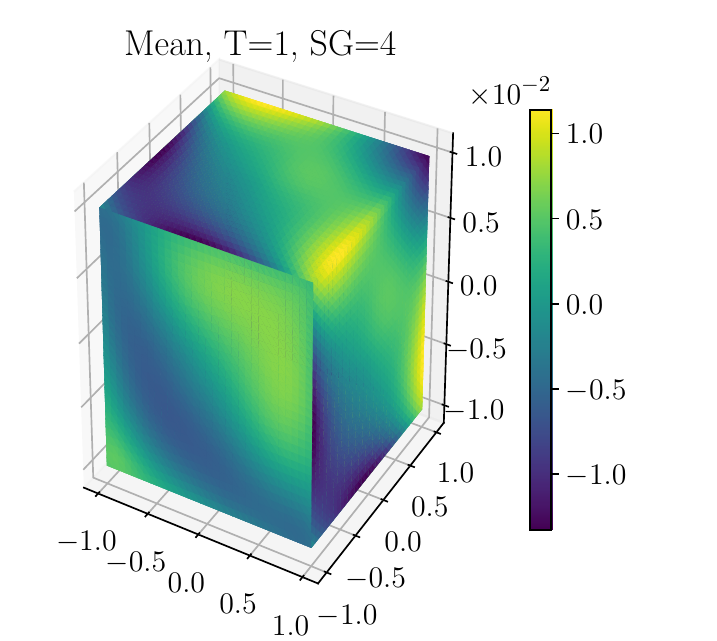}
 \includegraphics[trim=0mm 0mm 0mm 6mm, clip,width =0.5\linewidth]{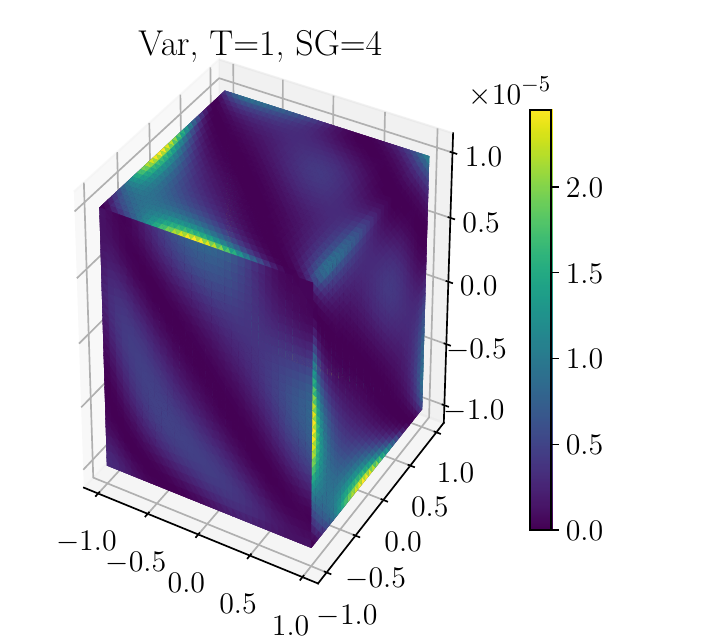}
 \caption{Plot of a) Mean and b) variance of the first component of the solution at time $T=1$, refinement level 5 and SG polynomial degree 4 and $\alpha \sim \mathcal{U}(0.1,1.9)$ on $\Gamma$.
}
\label{fig:UniformAcousticMeanAdVarianceT1}
\end{figure}

\begin{figure} 
  \centering
  \begin{tikzpicture}[scale=0.7]
  \begin{axis}[
  xmin = -0.5, xmax = 8,
  ymin = 1e-11, ymax = 1.0,
  ymode = log,
  ylabel = {relative $L^2((0,T)\times\Gamma,L^2(\Xi))$-error},
  xlabel = {SG degree},
  xmode = linear,
  grid = major,
  cycle list name = color,
  legend style={font=\scriptsize},
  legend style={at={(1.,1.)}}
  ]
  \addplot
      table[row sep=crcr]{%
      0		0.193755791203514 \\
      1			0.001284324221099\\
      2			3.4264184862695E-05\\
      3			1.516598523281E-06\\
      4			1.28905686377627E-07\\
      5			3.04259490466013E-08\\
      6			2.19136941930942E-08\\
      7			1.63035917103951E-10 \\     
    };
  \addlegendentry{First component, fifth refinement}
  \addplot
  table[row sep=crcr]{
  0			0.421963984113727\\
  1			0.000834879419778\\
  2			1.30987497787212E-05\\
  3			1.39566016215133E-07\\
  4			2.49445766403383E-09\\
  5			4.7701038504662E-09\\
  6			1.17100671215417E-09\\
  7			1.29354122191381E-09 \\ 
};
\addlegendentry{Second component, fifth refinement}
  \end{axis}
  \end{tikzpicture}
  \begin{tikzpicture}[scale=0.7]
    \begin{axis}[
    xmin = 0, xmax = 2000000,
    ymin = 0.001, ymax = 1,
    ymode = log,
    ylabel = {relative $L^2((0,T)\times\Gamma,L^2(\Xi))$-error},
    xlabel = {space-time DOF},
    xmode = log,
    grid = major,
    cycle list name = color,
    legend style={font=\scriptsize},
    legend style={at={(1.8,1.)}}
    ]
    \addplot
      table[row sep=crcr]{%
    222			0.249238298889772\\
    1740			0.095007099079354\\
    13848			0.028109401753582\\
    110640			0.00742652531148\\
    884832			0.001857034177089\\
    };
    \addlegendentry{First component, SG degree 0}
    \addplot
      table[row sep=crcr]{%
    222			0.255110253223579\\
    1740			0.087786638281467\\
    13848			0.024245790184526\\
    110640			0.006246793005604\\
    884832			0.001562037106462\\
    };
    \addlegendentry{First component, SG degree 2}
    \addplot
      table[row sep=crcr]{%
    222			0.255104186131423\\
    1740			0.087811374969219\\
    13848			0.02425276738158\\
    110640			0.006248627574157\\
    884832			0.001562495848115\\
    };
    \addlegendentry{First component, SG degree 4}
    
    \addplot
      table[row sep=crcr]{%
    222			0.338601875252293\\
    1740			0.207772559758001\\
    13848			0.055077265903734\\
    110640			0.01389982274242\\
    884832			0.003475709676537\\
    };
    \addlegendentry{Second component, SG degree 0}
    \addplot
      table[row sep=crcr]{%
    222			0.338497913951179\\
    1740			0.208962877864616\\
    13848			0.055362963430939\\
    110640			0.013968280541238\\
    884832			0.003492827839709\\
    };
    \addlegendentry{Second component, SG degree 2}
    \addplot
      table[row sep=crcr]{
    222			0.338493245151572\\
    1740			0.208971755956177\\
    13848			0.055364775885809\\
    110640			0.013968644946469\\
    884832			0.003492918960784\\
    };
    \addlegendentry{Second component, SG degree 4}
    \addplot [color=black,dashed, thick]
      table[row sep=crcr]{%
    222	0.172093585455314\\
    884832	0.000684583011673\\
    };
    \addlegendentry{order -2/3}
    \end{axis}
    \end{tikzpicture}
    \caption{Plot of relative $L^2((0,T)\times\Gamma,L^2(\Xi))$-errors for a) a fixed space-time mesh and b) a fixed SG polynomial degree and  $\alpha \sim \mathcal{U}(0.1,1.9)$ on $\Gamma$.} 
    \label{fig:UniformAcousticError}
  \end{figure}
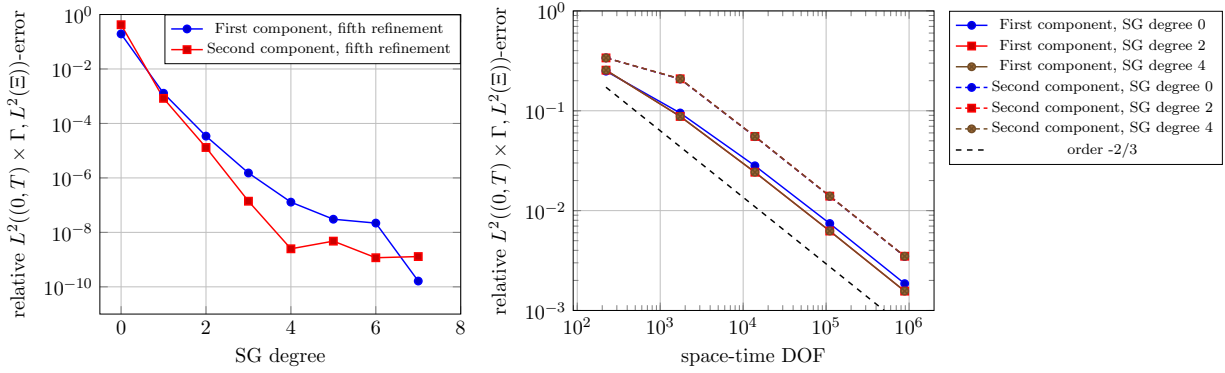

In our next numerical example, we let  $\alpha \sim \mathcal{U}(1.5,2.5)$ on the top of the cube $\Gamma_1 := [-1,1]^{2}\times \{1\}$ and 
and $\alpha \sim \mathcal{U}(3,4)$ on $\Gamma_2 := \Gamma \setminus \Gamma_1 $.
In \figref{fig:DiscAcousticErrorMeanAndVariance2} the mean and variance of $p$ are presented at time $T=1$ for the highest refinement level and SG polynomial degree 4. The plot clearly differs from \figref{fig:UniformAcousticMeanAdVarianceT1}, where the variance only has entries on the edges. Similar to \figref{fig:UniformAcousticMeanAdVarianceT1}, we observe the highest values on the edges for the mean as well as for the variance.

In \figref{fig:DiscAcousticError}, we plot an analoguous result to the \figref{fig:DiscAcousticError}. The first plot shows a reduction of the error of $10^{-3}$ until SG polynomial degree 2 for $\varphi$ and $10^{-2}$ until SG polynomial degree 3 for $p$. Then there is a very slow convergence until degree 6. For the second plot, we again observe a convergence rate of $-2/3$ in a fixed space-time DOF for $\varphi$. For $p$ we can't tell a clear rate, but it tends to be $-2/3$, too, except for the 3rd to the 4th refinement, like in the corresponding Figure before.

\begin{figure}
  \includegraphics[width =0.5\linewidth]{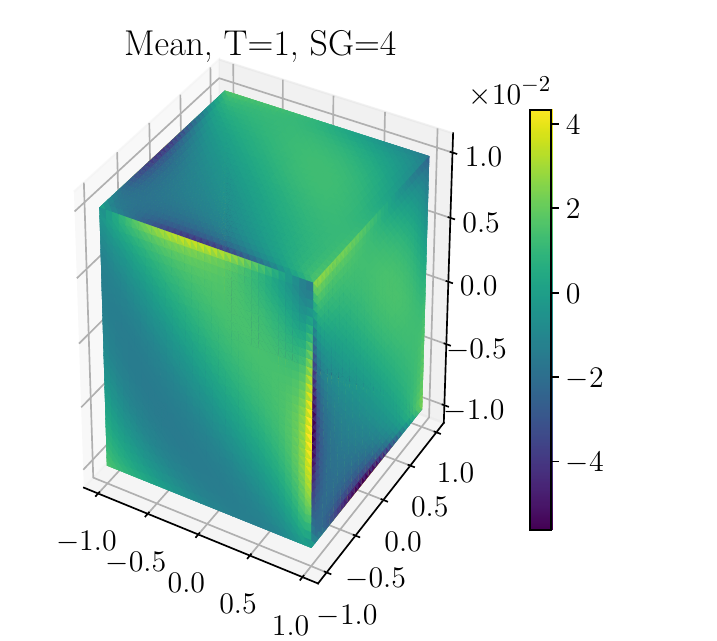}
 \includegraphics[width =0.5\linewidth]{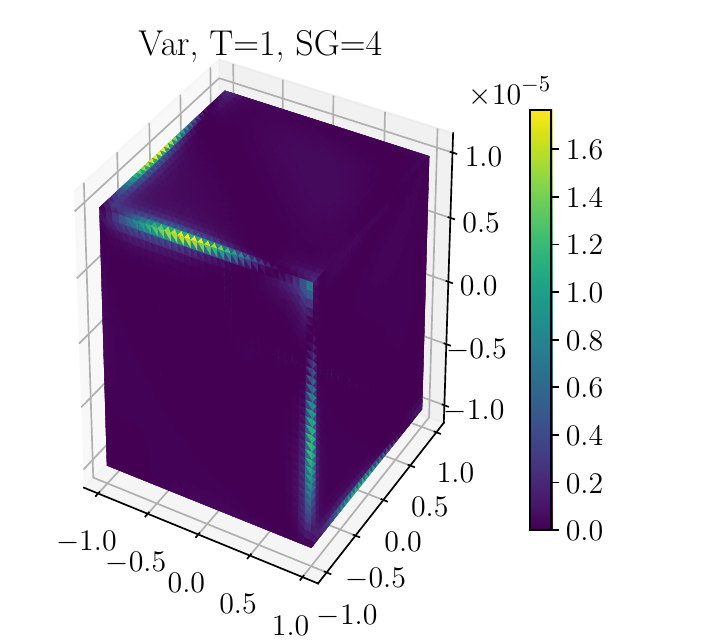}
 \caption{Plot of a) Mean and b) variance of the first component of the solution  at time $T=1$ for refinement level 5 and SG polynomial degree 4 and $\alpha \sim \mathcal{U}(1.5,2.5)$ on $\Gamma_1$ and $\alpha \sim \mathcal{U}(3,4)$ on $\Gamma_2$. }
\label{fig:DiscAcousticErrorMeanAndVariance2}
\end{figure}

\begin{figure}
\centering
  \begin{tikzpicture}[scale=0.7]
    \begin{axis}[
    xmin = -0.5, xmax = 8,
    ymin = 1e-12, ymax = 1.0,
    ymode = log,
    ylabel = {relative $L^2((0,T)\times\Gamma,L^2(\Xi))$-error},
    xlabel = {SG degree},
    xmode = linear,
    grid = major,
    cycle list name = color,
    legend style={font=\scriptsize},
    legend style={at={(1.,1.)}},
    ]
    \addplot
        table[row sep=crcr]{%
        0			0.006957470158886\\
        1			3.94490881653416E-05\\
        2			2.15132001636636E-07\\
        3			1.59284504100222E-09\\
        4			1.15587013948433E-09\\
        5			5.46926752084718E-10\\
        6			4.11356698099475E-10\\
        7			9.34752388292965E-12\\        
      };
    \addlegendentry{First component, fifth refinement}
    \addplot
    table[row sep=crcr]{
      0			0.00944517856858\\
      1			1.52990413016025E-05\\
      2			3.47075531898231E-08\\
      3			1.42373869704038E-08\\
      4			5.96468634600755E-09\\
      5			3.94107978265053E-09\\
      6			1.20540400614035E-09\\
      7			1.10432316379714E-12\\       
  };
  \addlegendentry{Second component, fifth refinement}
  \end{axis}
\end{tikzpicture} 
\begin{tikzpicture}[scale=0.7]
\begin{axis}[
xmin = 0, xmax = 2000000,
ymin = 0.00005, ymax = 1,
ymode = log,
ylabel = {relative $L^2((0,T)\times\Gamma,L^2(\Xi))$-error},
xlabel = {space-time DOF},
xmode = log,
grid = major,
cycle list name = color,
legend style={font=\scriptsize},
legend style={at={(1.8,1.)}}
]
\addplot
  table[row sep=crcr]{%
222			0.35347195710067\\
1740			0.11529950827877\\
13848			0.017807140903367\\
110640			0.00040595547975\\
884832			0.000101510890847\\
};
\addlegendentry{First component, SG degree 0}
\addplot
  table[row sep=crcr]{%
222			0.354602215425068\\
1740			0.116547340240718\\
13848			0.018515132960014\\
110640			0.000634160213476\\
884832			0.000158574453167\\
};
\addlegendentry{First component, SG degree 2}
\addplot
  table[row sep=crcr]{%
222			0.354602202944901\\
1740			0.116547342218061\\
13848			0.018515142802841\\
110640			0.000634176235601\\
884832			0.000158578459564\\
};
\addlegendentry{First component, SG degree 4}
\addplot
  table[row sep=crcr]{%
222			0.321904658298223\\
1740			0.164059971913243\\
13848			0.037617885213133\\
110640			0.008216363382759\\
884832			0.002054536539394\\
};
\addlegendentry{Second component, SG degree 0}
\addplot
  table[row sep=crcr]{%
222			0.320403016365026\\
1740			0.163523100067492\\
13848			0.03732521194525\\
110640			0.008103608637961\\
884832			0.002026341736852\\
};
\addlegendentry{Second component, SG degree 2}
\addplot
  table[row sep=crcr]{%
222			0.320402759799301\\
1740			0.163523025920587\\
13848			0.037325183841967\\
110640			0.008103598603982\\
884832			0.002026339227814\\
};
\addlegendentry{Second component, SG degree 4}
\addplot [color=black,dashed, thick]
  table[row sep=crcr]{%
222	0.172093585455314\\
884832	0.000684583011673\\
};
\addlegendentry{order -2/3}
\end{axis}
\end{tikzpicture}
\caption{Plot of relative $L^2((0,T)\times\Gamma,L^2(\Xi))$-errors for a) a fixed space-time mesh and b) a fixed SG polynomial degree with $\alpha \sim \mathcal{U}(1.5,2.5)$ on $\Gamma_1$ and $\alpha \sim \mathcal{U}(3,4)$ on $\Gamma_2$.
}
\label{fig:DiscAcousticError}
\end{figure}
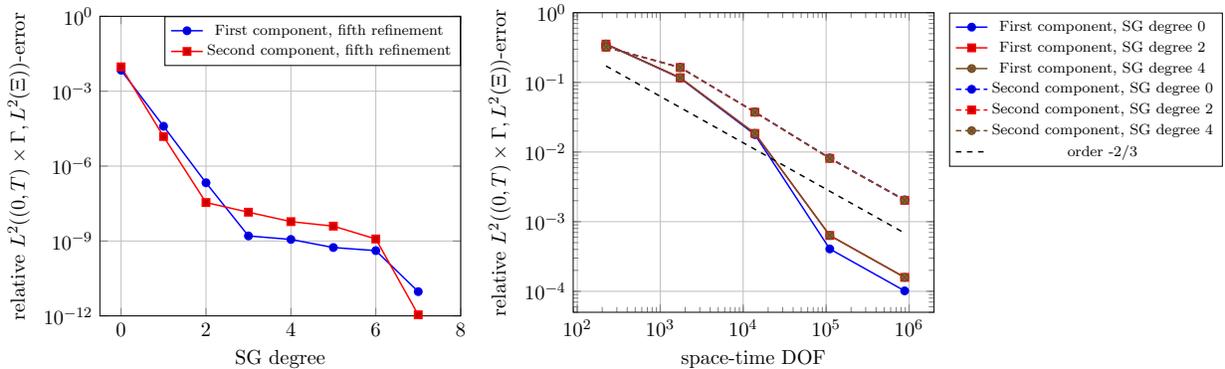

 In our final example we consider the stochastic coefficient of the impedance function 
 $\alpha \sim \mathcal{U}(0.1,0.9)$ on $\Gamma_1 := [-1,1]^{2} \times \{1\}$ and 
  $\alpha \sim \mathcal{U}(3,4)$ on $\Gamma_2 := \Gamma \setminus \Gamma_1 $.
	
In \figref{fig:DiscAcousticErrorMeanAndVariance_reducedConvergence}, we plot again the mean and the variance of $p$ with the same settings as in the corresponding plots before. We observe that indeed edge singularitutes of the mean is independent of the variance. As in \figref{fig:DiscAcousticErrorMeanAndVariance2} the interior of the boundary is zero on the whole boundary except the top and the edges. For the mean the change of $\alpha$ at the top causes that a high value switches into another edge. 

In \figref{fig:DiscAcousticError_reducedConvergence}a) correspondingly to the other Figures \ref{fig:DiscAcousticErrorMeanAndVariance2} and \ref{fig:UniformAcousticMeanAdVarianceT1}, we see an error reduction of $10^{-2}$ until degree $2$ for $p$ and $\varphi$. Afterwards they decay spectrally. 
For a fixed SG polynomial degree in \figref{fig:DiscAcousticError_reducedConvergence}b), we see for $\varphi$ a rate of $1/4$ and for $p$ a rate of $1$.

Altogether, we conclude that not only the convergence rates, also the edge and corner singularities depend highly on the impedance $\alpha$ despite choosing nonconforming ansatz and test functions.

  \begin{figure}
    \includegraphics[width =0.5\linewidth]{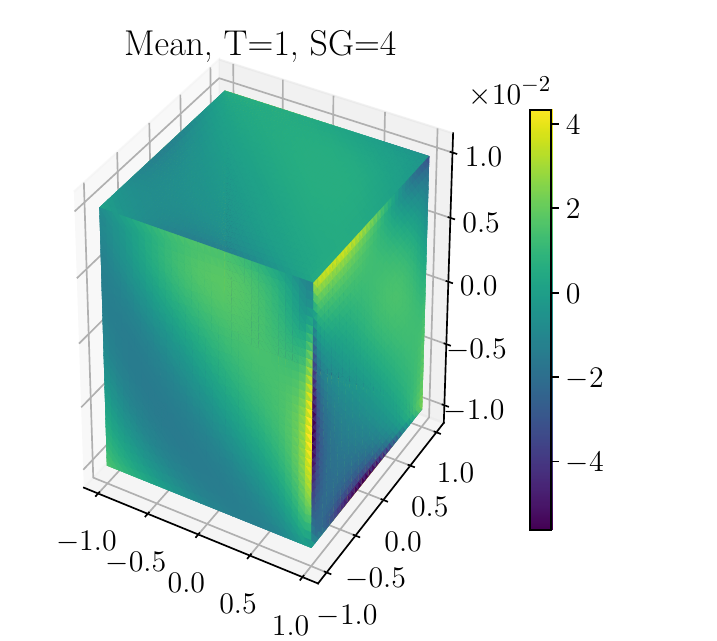}
   \includegraphics[width =0.5\linewidth]{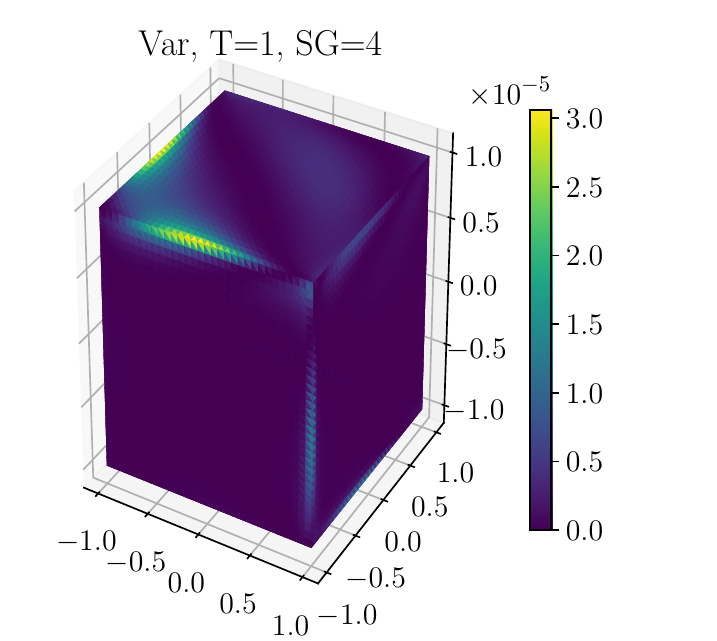}
   \caption{Plot of a) Mean and b) variance of the first component of the solution  at time $T=1$, refinement level 5 and SG polynomial degree 4 and $\alpha \sim \mathcal{U}(0.1,0.9)$ on $\Gamma_1$ and $\alpha \sim \mathcal{U}(3,4)$ on $\Gamma_2$. 
	}
  \label{fig:DiscAcousticErrorMeanAndVariance_reducedConvergence}
  \end{figure}

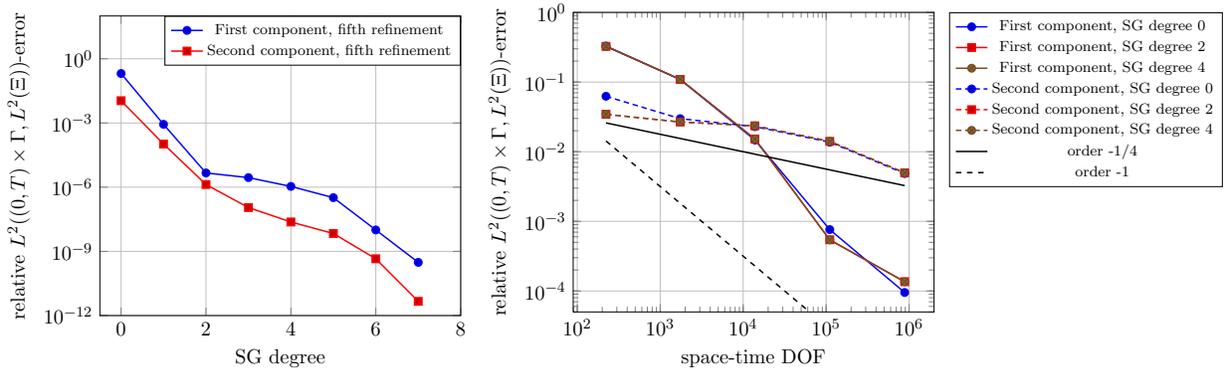
\begin{figure}
  \centering
    \begin{tikzpicture}[scale=0.7]
      \begin{axis}[
      xmin = -0.5, xmax = 8,
      ymin = 1e-12, ymax = 100.0,
      ymode = log,
      ylabel = {relative $L^2((0,T)\times\Gamma,L^2(\Xi))$-error},
      xlabel = {SG degree},
      xmode = linear,
      grid = major,
      cycle list name = color,
      legend style={font=\scriptsize},
      legend style={at={(1.,1.)}},
      ]
      \addplot
          table[row sep=crcr]{%
          0			0.205134846544482\\
          1			0.000864269288874\\
          2			4.61660003333125E-06\\
          3			2.79579882141877E-06\\
          4			1.08367003162943E-06\\
          5			3.23668455566752E-07\\
          6			1.00318911520655E-08\\
          7			3.0563252097638E-10\\                  
        };
      \addlegendentry{First component, fifth refinement}
      \addplot
      table[row sep=crcr]{
        0			0.010988953169125\\
        1			0.000103499170289\\
        2			1.32741146000021E-06\\
        3			1.11553985683177E-07\\
        4			2.36072126726895E-08\\
        5			6.83371518364325E-09\\
        6			4.50957089506952E-10\\
        7			4.62148867029359E-12\\                 
    };
    \addlegendentry{Second component, fifth refinement}
    \end{axis}
  \end{tikzpicture} 
  \begin{tikzpicture}[scale=0.7]
  \begin{axis}[
  xmin = 0, xmax = 2000000,
  ymin = 0.00005, ymax = 1,
  ymode = log,
  ylabel = {relative $L^2((0,T)\times\Gamma,L^2(\Xi))$-error},
  xlabel = {space-time DOF},
  xmode = log,
  grid = major,
  cycle list name = color,
  legend style={font=\scriptsize},
  legend style={at={(1.8,1.0)}}
  ]
  \addplot
    table[row sep=crcr]{%
    222			0.329116613417176\\
    1740			0.109011713806115\\
    13848			0.014703738597195\\
    110640			0.000761160582844\\
    884832			9.51760411986228E-05\\     
  };
  \addlegendentry{First component, SG degree 0}
  \addplot
    table[row sep=crcr]{%
    222			0.324899486245871\\
    1740			0.10903839949066\\
    13848			0.015224439686129\\
    110640			0.000544017943584\\
    884832			0.000136033995956\\    
  };
  \addlegendentry{First component, SG degree 2}
  \addplot
    table[row sep=crcr]{%
    222			0.324895701012458\\
    1740			0.109041819545805\\
    13848			0.015227295980278\\
    110640			0.000541248924942\\
    884832			0.000135341591091\\    
  };
  \addlegendentry{First component, SG degree 4}
  \addplot
    table[row sep=crcr]{%
    222			0.062515893062138\\
    1740			0.029809849235324\\
    13848			0.022897717950623\\
    110640			0.013729889893706\\
    884832			0.00485503905401\\
  };
  \addlegendentry{Second component, SG degree 0}
  \addplot
    table[row sep=crcr]{%
    222			0.034492143132376\\
    1740			0.026706957467035\\
    13848			0.023522679302499\\
    110640			0.014068682016589\\
    884832			0.004974839649683\\    
  };
  \addlegendentry{Second component, SG degree 2}
  \addplot
    table[row sep=crcr]{%
    222			0.034427911537667\\
    1740			0.026719661579596\\
    13848			0.023521820373035\\
    110640			0.014068570610937\\
    884832			0.004974800255427\\    
  };
  \addlegendentry{Second component, SG degree 4}
  \addplot [color=black, thick]
    table[row sep=crcr]{%
    222	0.025906679741218\\
    884832	0.003260504431908\\       
  };
  \addlegendentry{order -1/4}
  \addplot [color=black,dashed, thick]
  table[row sep=crcr]{%
  222	0.014244493964723\\
  884832	3.57387352646422E-06\\    
  };
  \addlegendentry{order -1}
  \end{axis}
  \end{tikzpicture}
  \caption{Plot of relative $L^2((0,T)\times\Gamma,L^2(\Xi))$-errors for a) a fixed space-time mesh and b) a fixed SG polynomial degree for $\alpha \sim \mathcal{U}(0.1,0.9)$ on $\Gamma_1$ and $\alpha \sim \mathcal{U}(3,4)$ on $\Gamma_2$.
	}
  \label{fig:DiscAcousticError_reducedConvergence}
  \end{figure}

\subsection{{Application to traffic noise}}
\label{app:traficnoise}
{We finally indicate the feasibility and relevance of the developed techniques for problems in traffic noise. We consider the sound impact of a grown slick 205/55R16 passenger car tyre, of diameter around $0.6$m at $2$ bar pressure and subject to $3415$N axle load at $50$km/h on a flat street with an ISO~10844 surface. The domain is here given by the complement of the tyre in the upper half space, with the origin in the centre of the tyre, and its boundary $\Gamma$ is shown in Figure \ref{fig:StreetTyreMesh}. The sound pressure evolves according to a wave equation with wave speed $c=343$m/s.}

{Following \cite{comput}, the sound emission corresponds to Neumann boundary condition \eqref{eq:stochNeumann} with right hand side $f = -\rho \frac{\partial^2 u_n}{\partial t^2}$, which has been obtained from simulations of the tyre displacement $u_n$ normal to the surface. For simplicity, we here consider $f$ to be deterministic. On the street surface acoustic boundary conditions \eqref{eq:stochAcoustic} with right hand side $f=0$ are imposed.
The coefficient $\alpha$ depends strongly on the characteristics of the street surface \cite{de2000trias}, and it is therefore of interest to study the sound emission for a realistic range $\alpha \sim \mathcal{U}(1,10)$ of its values. We extend $\alpha$ by $0$ to the surface of the tyre. }

{As in \cite{comput,ha1995}, we use a single layer ansatz for the sound pressure, leading to the second-kind boundary integral equation
\begin{equation} \label{2ndkind}
(-\textstyle{\frac{1}{2}} I+\mathcal{K}')\phi-\alpha \mathcal{V}\phi=f.
\end{equation}   
The sound pressure is recovered from $\phi$ in a postprocessing step. Note that this formulation \eqref{2ndkind} is commonly used by engineers because, unlike \eqref{eq:acoustic_recall}, it only involves the two operators $\mathcal{V}$ and $\mathcal{K}'$. However, the mathematical analysis even of deterministic second-kind boundary integral equations remains widely open for time-dependent problems.}

{The truncated boundary $\Gamma$ is triangulated by the mesh in Figure \ref{fig:StreetTyreMesh} with 14092 elements. We set $\Delta t=0.01/343\text{s} \simeq 2.915 \times 10^{-5}$s and $T=1/343\text{s} \simeq 2.915 \times 10^{-3}$s. We discretize  the weak formulation of \eqref{2ndkind} on this mesh by  piecewise constant ansatz and test functions in space and time, and stochastic polynomial degree 1. Details are provided in Appendix \ref{appendix:disctraffic}.  }

\begin{figure}
\includegraphics[trim=75mm 35mm 75mm 95mm, clip,width=0.98 \linewidth]{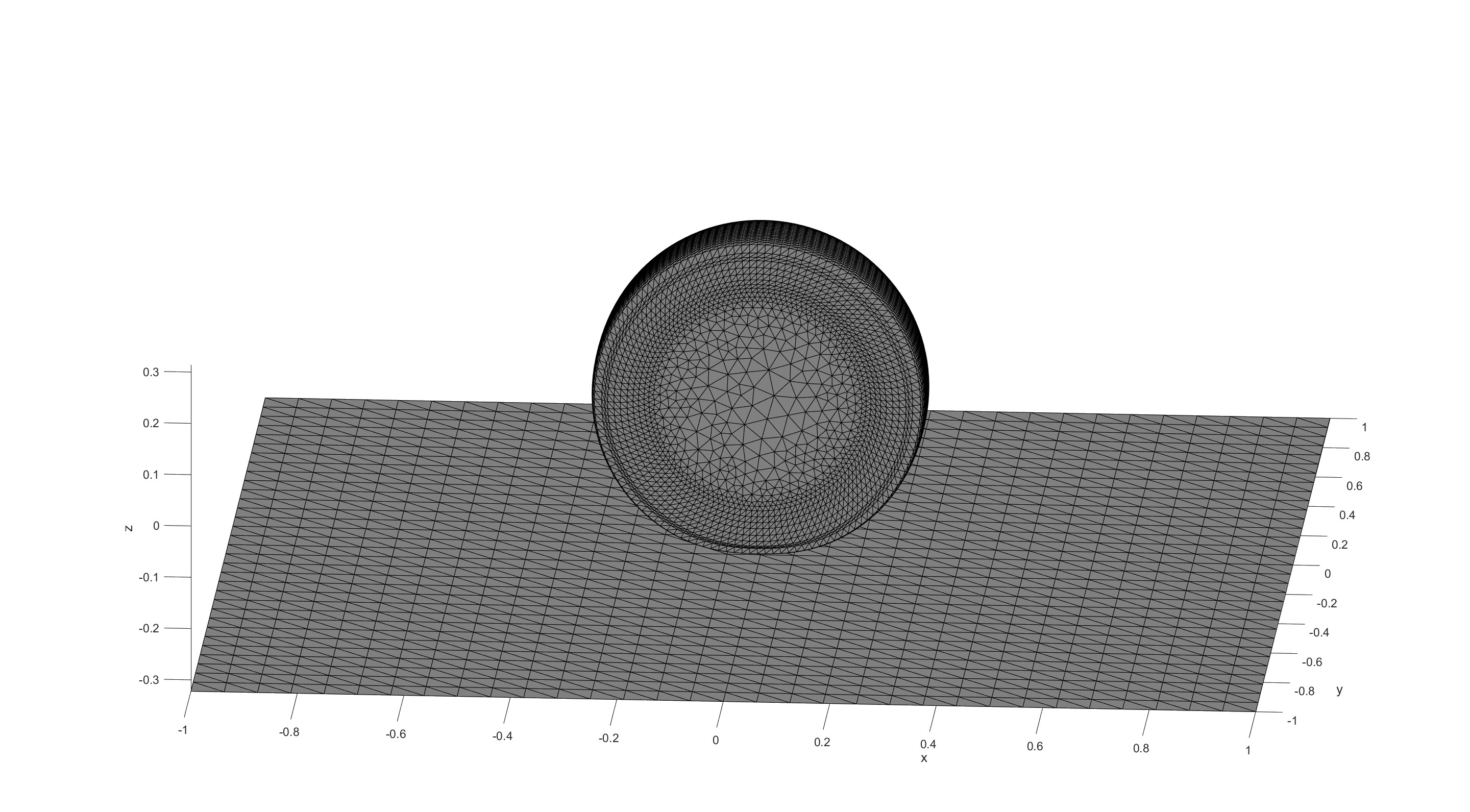}
  \caption{Geometry in  
  \secref{app:traficnoise}.}
  \label{fig:StreetTyreMesh}
\end{figure}

  \begin{figure}
    \includegraphics[trim=65mm 32mm 55mm 95mm, clip,width =0.98\linewidth]{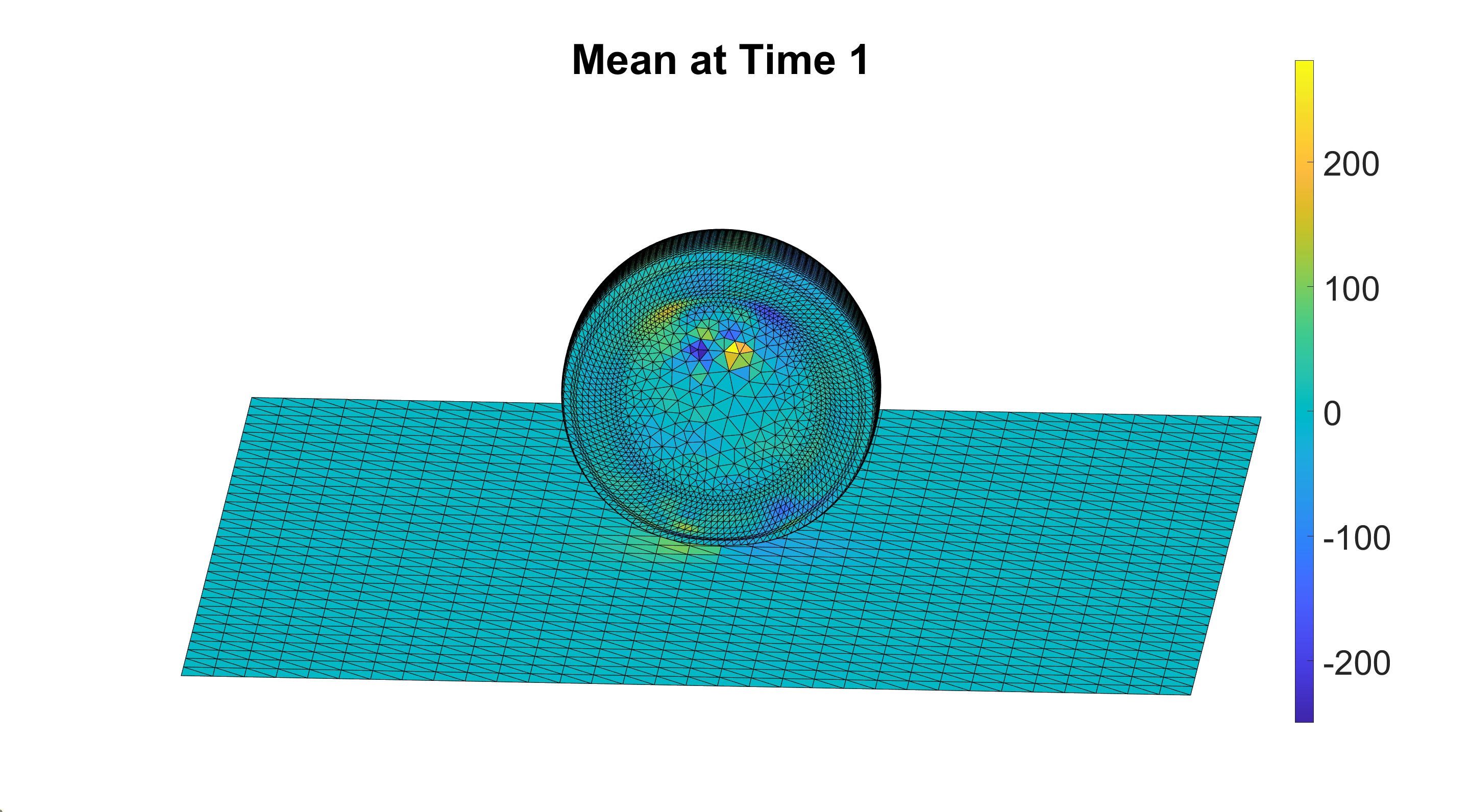}
   \includegraphics[trim=65mm 32mm 55mm 95mm, clip,width =0.98\linewidth]{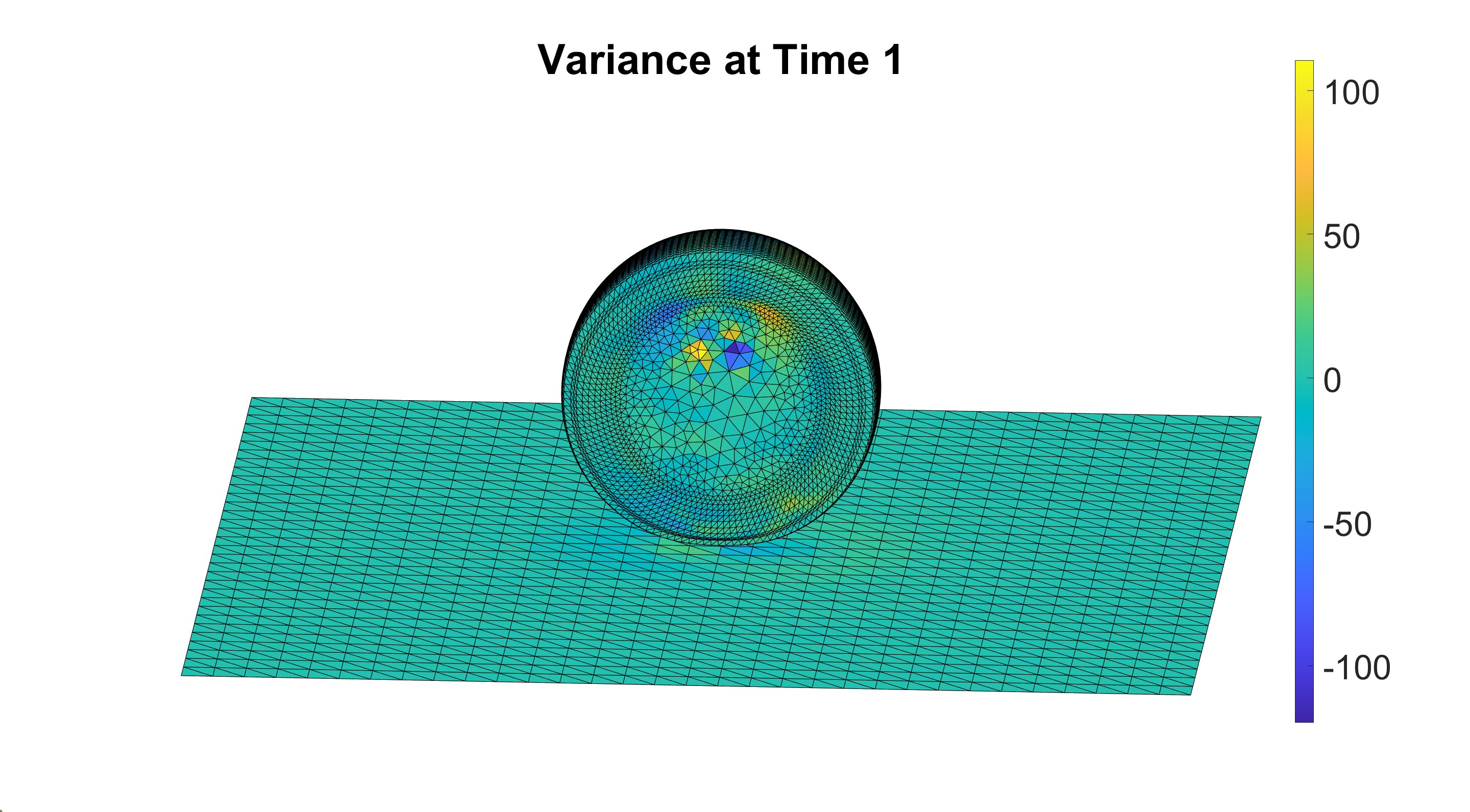}
  \caption{Plot of a) the mean and b) the variance of the density at time $T=2.915 \times 10^{-3}$s. }
  \label{fig:AcousticMeanAndVariance_Street_Tyre}
 \end{figure}


{In Figure \ref{fig:AcousticMeanAndVariance_Street_Tyre}, we plot the mean and the variance of the density $\phi$. As the source term $f$ is nonzero only on the tyre, the density vanishes on the street outside a neighborhood of size $cT$ of the tyre.}

{After postprocessing the density $\phi$ using the single layer potential, Figure \ref{fig:Eval_0_125_0_-0_31} depicts the sound pressure in a point $(0.125\text{m},0\text{m},-0.31\text{m})$ near the street surface close to the tyre. }

{
The results in the time domain also provide an efficient way to obtain information for a wide band of frequencies. The A-weighted sound pressure level gives an approximation
to the human perception of noise \cite{hoever2012influence}. Figure \ref{fig:fftA}
shows the A-weighted mean sound pressure level in dB of the acoustic
wave emitted from the tyre, averaged
over $17$ points at a distance of $0.4$m from the center, as well as one standard deviation. It is obtained by a
discrete fast Fourier transform in time from the sound pressure. 
}

{Simulations as above of the sound impact of a passenger car tyre provide the computional basis for investigations in traffic noise \cite{comput}. For realistic problems, both sources and the properties of the street surface are uncertain, and the proposed methods allow to quantify  the mean sound pressure level and its variance as key ingredients for further analyses.   }

  \begin{figure}
    \centering\includegraphics[width =0.7\linewidth]{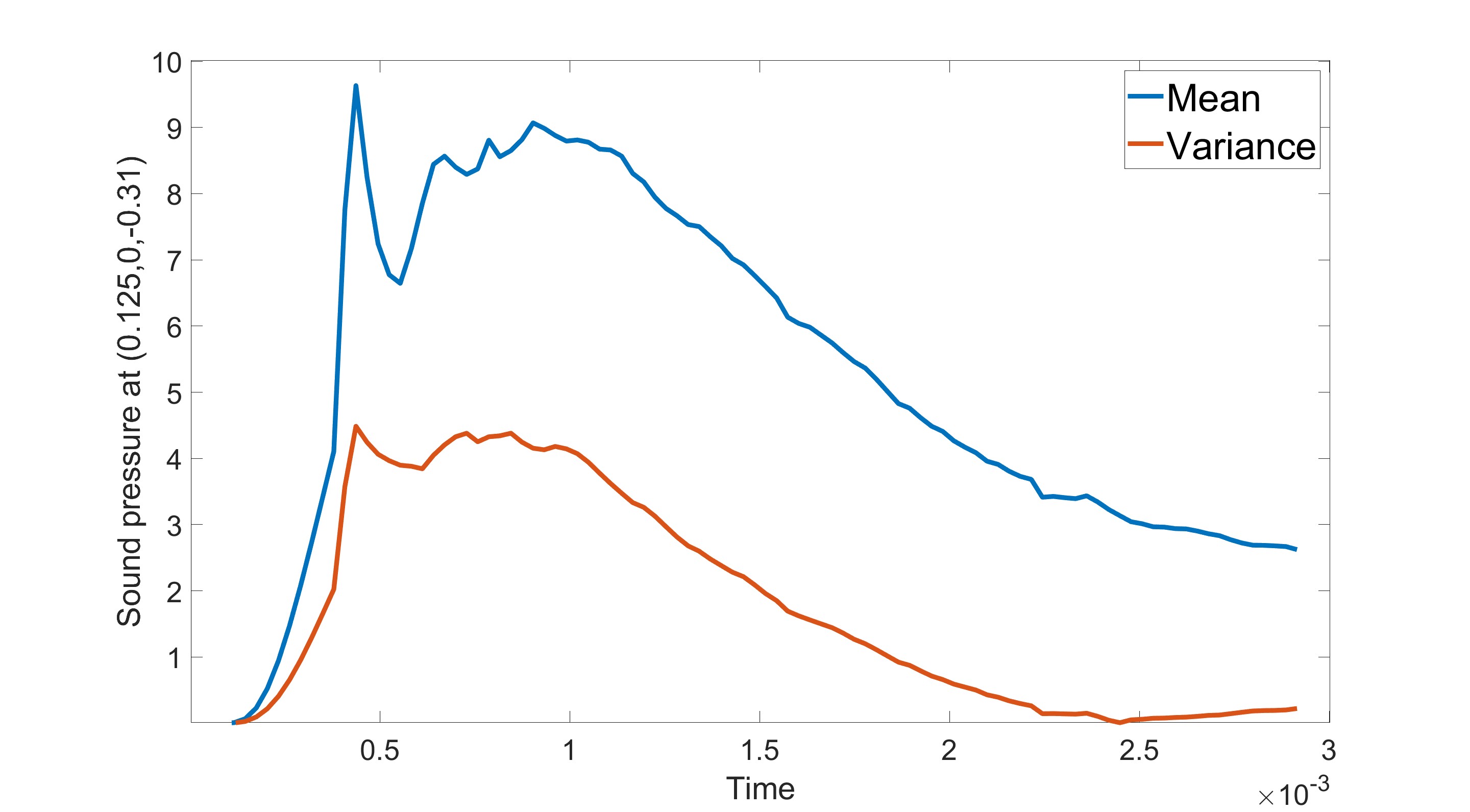}
   \caption{Mean and variance of the sound pressure in the point $(0.125\text{m},0\text{m},-0.31\text{m})$. 
	}
  \label{fig:Eval_0_125_0_-0_31}
%
    \centering\includegraphics[width =0.7\linewidth]{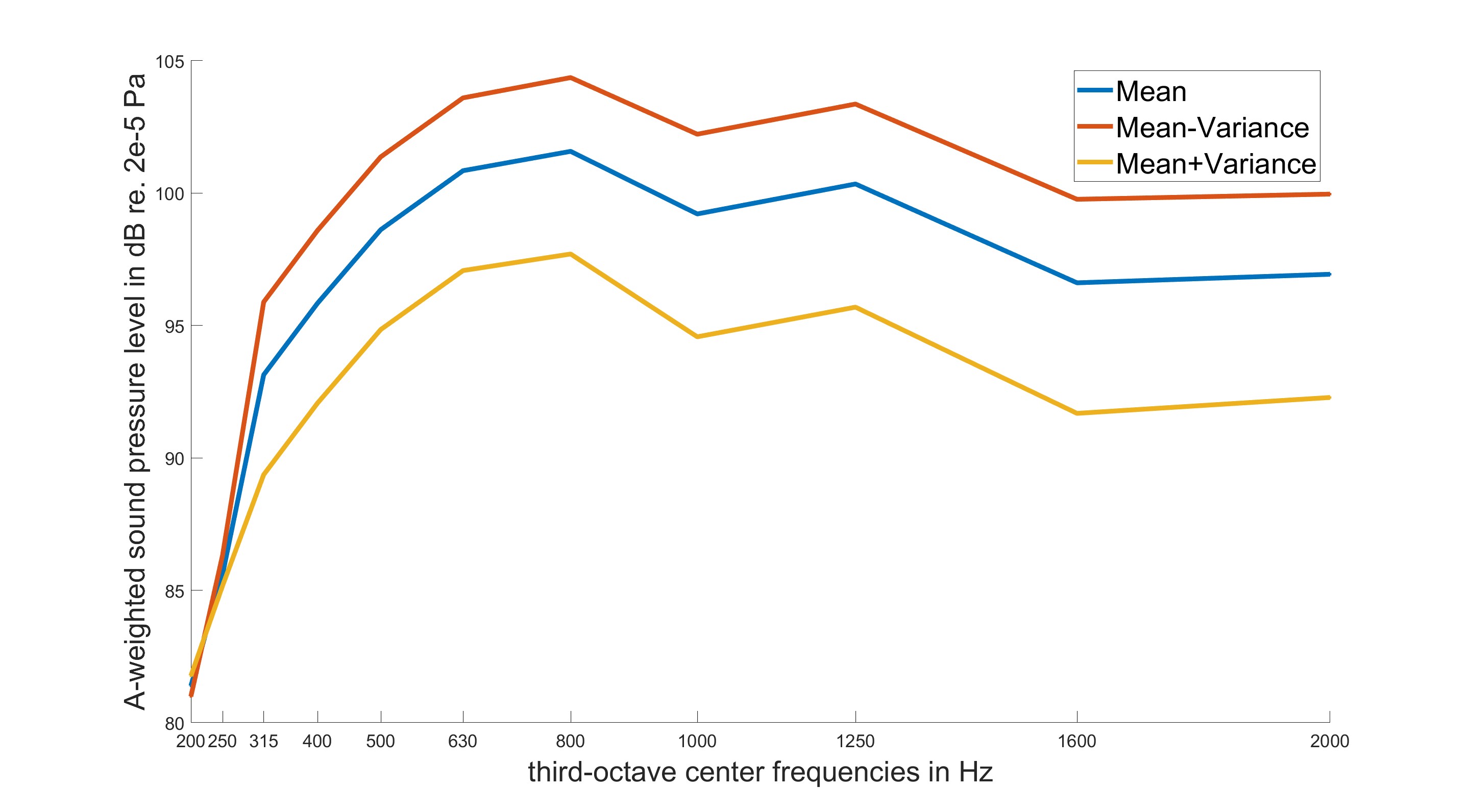}
   \caption{A-weighted mean sound pressure and confidence intervall as function of frequency, averaged over $17$ points. 
	}
  \label{fig:fftA}

  \end{figure}



\section{Conclusions}

In this article   we have introduced and studied a stochastic space-time Galerkin boundary element method for the acoustic wave equation. Stochastic Dirichlet, Neumann and acoustic boundary conditions are considered, with uncertain sources and acoustic parameters. The approach builds on recent advances for time domain boundary integral equations.

After a polynomial chaos expansion of the random variables, we obtained a high-dimensional space-time boundary integral equation. This integral equation is discretized using a Galerkin method based on tensor products of piecewise polynomial basis functions in space, respectively in time, leading to a  high-dimensional marching-on-in-time time-stepping scheme \cite{terrasse}. After presenting the formulation and its theoretical background, we discussed the detailed discretization in the stochastic and space-time variables. The numerical experiments for the wave equation in $3d$ empirically confirmed the performance and the convergence rates for increasing stochastic, respectively space-time degrees of freedom. {They also confirmed the applicability of the proposed methods for applications to traffic noise.}

\bibliographystyle{siam}
\bibliography{bibliography}

\appendix

\section{Set-up of space-time anisotropic Sobolev spaces}

Space--time anisotropic Sobolev spaces provide a convenient setting to study time dependent boundary integral equations like \eqref{symmhyp} and \eqref{hypersingeq}. We refer to \cite{setup,hd} for a detailed exposition. 

{To define the relevant class of function spaces, if $\partial\Gamma\neq \emptyset$, we first extend $\Gamma$ to a closed, orientable Lipschitz manifold $\widetilde{\Gamma}$.  On $\Gamma$ one defines the usual Sobolev spaces of supported distributions:
$$\widetilde{H}^s(\Gamma) = \{u\in H^s(\widetilde{\Gamma}): \mathrm{supp}\ u \subset {\overline{\Gamma}}\}\ , \quad\ s \in \mathbb{R}\ .$$
Furthermore, ${H}^s(\Gamma)$ is the quotient space $ H^s(\widetilde{\Gamma}) / \widetilde{H}^s({\widetilde{\Gamma}\setminus\overline{\Gamma}})$. \\
{To write down an explicit family of Sobolev norms, introduce a partition of unity $\alpha_i$ subordinate to a covering of $\widetilde{\Gamma}$ by open sets $B_i$. For diffeomorphisms $\varphi_i$ mapping each $B_i$ into the unit cube $\subset \mathbb{R}^n$, a family of Sobolev norms is induced from $\mathbb{R}^n$, with parameter $\omega \in \mathbb{C}\setminus \{0\}$:
\begin{equation*}
 ||u||_{s,\omega,{\widetilde{\Gamma}}}=\left( \sum_{i=1}^p \int_{\mathbb{R}^n} (|\omega|^2+|\randomvar|^2)^s|\mathcal{F}\left\{(\alpha_i u)\circ \varphi_i^{-1}\right\}(\randomvar)|^2 d\randomvar \right)^{\frac{1}{2}}\ .
\end{equation*}
The norms for different $\omega \in \mathbb{C}\setminus \{0\}$ are equivalent and $\mathcal{F}$ denotes the Fourier transform. They induce norms on $H^s(\Gamma)$, $||u||_{s,\omega,\Gamma} = \inf_{v \in \widetilde{H}^s(\widetilde{\Gamma}\setminus\overline{\Gamma})} \ ||u+v||_{s,\omega,\widetilde{\Gamma}}$ and on $\widetilde{H}^s(\Gamma)$, $||u||_{s,\omega,\Gamma, \ast } = ||e_+ u||_{s,\omega,\widetilde{\Gamma}}$. We write $H^s_\omega(\Gamma)$ for $H^s(\Gamma)$, respectively $\widetilde{H}^s_\omega(\Gamma)$ for $\widetilde{H}^s(\Gamma)$, when a norm with a specific $\omega$ is fixed. $e_+$ extends the distribution $u$ by $0$ from $\Gamma$ to $\widetilde{\Gamma}$. As the norm $||u||_{s,\omega,\Gamma, \ast }$ corresponds to extension by zero, while $||u||_{s,\omega,\Gamma}$ allows  extension by an arbitrary $v$, $||u||_{s,\omega,\Gamma, \ast }$ is stronger than $||u||_{s,\omega,\Gamma}$. 
We now define a class of space-time anisotropic Sobolev spaces:
\begin{definition}\label{sobdef}
For {$\sigma>0$ and} $r,s \in\mathbb{R}$ define
\begin{align*}
 H^r_\sigma(\mathbb{R}^+,{H}^s(\Gamma))&=\{ u \in \mathcal{D}^{'}_{+}(H^s(\Gamma)): e^{-\sigma t} u \in \mathcal{S}^{'}_{+}(H^s(\Gamma))  \textrm{ and }   ||u||_{r,s,\sigma,\Gamma} < \infty \}\ , \\
 H^r_\sigma(\mathbb{R}^+,\widetilde{H}^s({\Gamma}))&=\{ u \in \mathcal{D}^{'}_{+}(\widetilde{H}^s({\Gamma})): e^{-\sigma t} u \in \mathcal{S}^{'}_{+}(\widetilde{H}^s({\Gamma}))  \textrm{ and }   ||u||_{r,s,\sigma,\Gamma, \ast} < \infty \}\ .
\end{align*}
$\mathcal{D}^{'}_{+}(E)$ respectively~$\mathcal{S}^{'}_{+}(E)$ denote the spaces of distributions, 
respectively tempered distributions, on $\mathbb{R}$ with support in $[0,\infty)$, 
taking values in a Hilbert space $E$. Here we consider $E= {H}^s({\Gamma})$, 
respectively  $E=\widetilde{H}^s({\Gamma})$. The relevant norms are given by
\begin{align*}
\|u\|_{r,s,\sigma}:=\|u\|_{r,s,\sigma,\Gamma}&=\left(\int_{-\infty+i\sigma}^{+\infty+i\sigma}|\omega|^{2r}\ \|\hat{u}(\omega)\|^2_{s,\omega,\Gamma}\ d\omega \right)^{\frac{1}{2}}\ ,\\
\|u\|_{r,s,\sigma,\ast} := \|u\|_{r,s,\sigma,\Gamma,\ast}&=\left(\int_{-\infty+i\sigma}^{+\infty+i\sigma}|\omega|^{2r}\ \|\hat{u}(\omega)\|^2_{s,\omega,\Gamma,\ast}\ d\omega \right)^{\frac{1}{2}}\,.
\end{align*}
\end{definition}
They are Hilbert spaces, and we note that the basic case $r=s=0$ is the weighted $L^2$-space 
with scalar product $\int_0^\infty e^{-2\sigma t} \int_\Gamma u \overline{v} ds_x\ dt$. 
Because $\Gamma$ is Lipschitz, like in the case of standard Sobolev spaces} 
these spaces are independent of the choice of $\alpha_i$ and $\varphi_i$ when $|s|\leq 1$.}\\

  \section{Algorithmic aspects of the stochastic problem}

\subsection{Discretization of the Dirichlet-Problem}\label{appendix:discstochdir}

Let $\randomvar=(\randomvar_{1},\ldots,\randomvar_{n})$, $x=(x_{1},x_{2},x_{3})$ with $N_{o}$ the amount of timesteps, $N_{s}$ the amount of elements and $\SGOrder$ the polynomial degree of Legendre polynomials, which are uniformly distributed at $[-1,1]$. Then we set the ansatz functions

$$
\dot{\phi}(t,x,\randomvar)=\sum_{i=0}^{(\SGOrder+1)^n-1} \dot{\phi}_{i}(t,x) \Psi(\randomvar) 
$$
with $\Psi_{i}(\randomvar)=P_{i_{1}}(\randomvar_{1}) \ldots P_{i_{n}}(\randomvar_{n})$, $0\leq i_{1},\ldots,i_{n}\leq \SGOrder$
and $\dot{\phi}_{i}(t,x)=\sum_{m=1}^{N_{o}} \sum_{s=1}^{N_{s}} \dot{\gamma}_{\Delta t}^{m}(t) \eta_{h}^{s}(x) \phi_{i}^{m,s}$,
 whereas $\gamma_{\Delta t}^{m}(t)$ resp. $\eta_{h}^{s}(x)$ are piecewise constant functions in time resp. in space $\Gamma$.

For the test functions, we choose $\psi(t,x,\randomvar)=\psi(t,x) \Psi_{j}(\randomvar)=\gamma_{\Delta t}^{l}(t) \eta_{h}^{r}(x) \Psi_{j}(\randomvar)$, $l=1,\ldots,N_{o}$, $r=1,\ldots,N_{o}$, $j=0,\ldots,(\SGOrder+1)^{n}-1$ with $0 \leq j_{1},\ldots,j_{n}\leq \SGOrder$.

Further we set $\sigma=0$, then we get the system:

\begin{align*}
  \int_{-1}^{1} \ldots \int_{-1}^{1}   
	\int_{0}^{\infty} \int_{\Gamma} \mathcal{V}(\sum_{i=0}^{(\SGOrder+1)^{n}-1} \sum_{m=1}^{N_{o}}\sum_{s=1}^{N_{s}} \dot{\gamma}_{\Delta t}^{m}(t) \eta_{h}^{s}(x) \Psi_{i}(\randomvar)\phi_{i}^{m,s}) \gamma^{l}_{\Delta t}(t) \eta_{h}^{r}(x) \Psi_{j}(\randomvar) ds_x dt d\randomvar_{n} \ldots d\randomvar_{1} (\tfrac{1}{2})^{n}
\end{align*}
Using the linearity of $\mathcal{V}$, we can separate the integral into the stochastic and space-time part.
Furthermore, since $\int_{-1}^{1} P_{i_{\tilde{k}}}(\randomvar_{\tilde{k}})P_{j_{\tilde{k}}}(\randomvar_{\tilde{k}}) d\randomvar_{\tilde{k}} = \begin{cases} \tfrac{2}{2 i_{\tilde{k}}+1} & i_{\tilde{k}}=j_{\tilde{k}} \\ 0 & \text{else} \end{cases}$ for $\tilde{k}=1,\ldots,k$,
considering the stochastic integrals we only need to look at the case $i=j$, where we nummerate $i=i_{1} \SGOrder^0 + i_{2} \SGOrder^{1} + \ldots i_{k} \SGOrder^{k} $:
\begin{align*}
 \int_{-1}^{1} P_{i_{1}}(\randomvar_{1}) P_{j_{1}}(\randomvar_{1}) d\randomvar_{1} \ldots \int_{-1}^{1} P_{i_{n}}(\randomvar_{n}) P_{j_{n}}(\randomvar_{n}) d\randomvar_{n} \tfrac{1}{2^{n}}
 = \tfrac{1}{(2 i_{n}+1) \dots (2 i_{1}+1)}
\end{align*}

Now let us consider the remaining space time integral.
\begin{align*}
  &\int_{0}^{\infty} \int_{\Gamma} \int_{\Gamma} \frac{(\sum_{m=1}^{N_{o}}\sum_{s=1}^{N_{s}} \dot{\gamma}_{\Delta t}^{m}(t-|x-y|) \eta_{h}^{s}(y) \phi_{i}^{m,s}) \gamma^{l}_{\Delta t}(t) \eta_{h}^{r}(x)}{4 \pi |x-y|} ds_y ds_x dt \\ &= 
	\sum_{m=1}^{N_{o}}\sum_{s=1}^{N_{s}} \phi_{i}^{m,s} \int_{\Gamma} \int_{\Gamma} \frac{\eta_{h}^{s}(y) \eta_{h}^{r}(x)}{4 \pi |x-y|}  \int_{0}^{\infty} \dot{\gamma}_{\Delta t}^{m}(t-|x-y|) \gamma^{l}_{\Delta t}(t) dt ds_y ds_x \\
	&= \sum_{m=1}^{N_{o}}\sum_{s=1}^{N_{s}} \phi_{i}^{m,s} \int_{\Gamma} \int_{\Gamma} \frac{\eta_{h}^{s}(y) \eta_{h}^{r}(x)}{4 \pi |x-y|} (-\chi_{E_{l-m-1}}(x,y)+\chi_{E_{l-m}}(x,y)) ds_y ds_x , 
\end{align*}
where $\chi_{E_{l-m}}$ is the indicator function with the lightcones $E_{l-m}=\lbrace (x,y) \in \Gamma\times\Gamma : t_{l-m}\leq|x-y|\leq t_{l-m+1}\rbrace$. 	
Defining the matrix $V^{l-m}$ with entries
$$
V^{l-m}_{(r,s)} = - \iint_{E_{l-m-1}} \frac{\eta_{h}^{s}(y) \eta_{h}^{r}(x)}{4 \pi |x-y|} ds_y ds_x + \iint_{E_{l-m}} \frac{\eta_{h}^{s}(y) \eta_{h}^{r}(x)}{4 \pi |x-y|} ds_y ds_x ,
$$
we get the space time system
\begin{align*} V \phi_{i} :=
\begin{pmatrix} {V^0} & 0&0 & 0&\cdots\\
   {V^1} & {V^0}  & 0& 0&\\
   {V^2} & {V^1} & {V^0}  &0 & \\
   V^3 & {V^2} & {V^1}  & {V^0} &\cdots \\  
   \vdots & & & \vdots& \ddots
   \end{pmatrix}
 \begin{pmatrix}  \phi^0_{i} \\ \phi^1_{i} \\ \phi^2_{i}\\ \phi^{3}_{i} \\
   \vdots 
   \end{pmatrix} =  \begin{pmatrix} F^0_{i} \\ F^1_{i} \\  F^2_{i} \\ F^{3}_{i} \\
   \vdots 
   \end{pmatrix} =: F_{i}
\end{align*}

Altogether we get the stochastic space time system:
\begin{equation}\label{eq:appstochspactimsys}
diag(\tfrac{1}{(2 i_{n}+1) \dots (2 i_{1}+1)} V ) (\phi_{0},\ldots,\phi_{i},\ldots,\phi_{(\SGOrder+1)^n-1})^{T} = F
\end{equation}

Now we are able to solve this system by applying forward substitution (marching on in time scheme) $(\SGOrder +1)^{n}$ times. For every $l$ and $i=0,\ldots,(\SGOrder+1)^{n}-1$ solve 
\begin{equation}\label{eq:appmotdirstoch}
  V^{0} \phi_{i}^{l} = F_{i}^{l} - \sum_{m=1}^{l-1} V^{l-m} \phi_{i}^{m}
\end{equation}

\subsection{Discretization of the stochastic acoustic problem}\label{appendix:discstochacoustic}
In this subsection, we write more details about the computation of the stochastic acoustic problem, in particular \eqref{eq:bilinear_recall}.

The coefficients $\alpha$, $\alpha^{-1}$ depend on random variables. Therefore the acoustic boundary conditions the system of equations leads no longer into a block diagonal system. However, the implementation of the stochastic part is independent of the implementation of the integral operators  $\mathcal{V}$, $\mathcal{K}$, $\mathcal{K}'$, $\mathcal{W}$ and may be done separately. 

For the retarded potential operators, we get similar results as before. 
However we have to consider the identity part with care. 

First we apply integration by parts on \eqref{eq:bilinear_recall} and divide the bilinearform into:

\begin{align}\label{eq:appendixacosutic}
  &\int_{-1}^{1}\ldots\int_{-1}^{1} \big( \int_{0}^{T} \int_{\Gamma} \tfrac{1}{\alpha} p q - 2 \mathcal{V} p \dot{q} + 2 \mathcal{K} \varphi \dot{q} + 2 \mathcal{K}' p \dot{\psi} - 2 \mathcal{W} \varphi \dot{\psi} + \alpha \dot{\varphi} \dot{\psi} ds_x dt \big) \omega(\xi) d\randomvar_{n} \ldots d\randomvar_{1} \nonumber
\\ &= \big( \int_{-1}^{1}\ldots\int_{-1}^{1} \int_{0}^{T} \int_{\Gamma} \tfrac{1}{\alpha} p q ds_x dt \omega(\xi) d\randomvar_{n} \ldots d\randomvar_{1} \big) \nonumber
 \\ &+ \big( \int_{-1}^{1}\ldots\int_{-1}^{1} \int_{0}^{T} \int_{\Gamma} \alpha \dot{\varphi} \dot{\psi} ds_x dt \omega(\xi) d\randomvar_{n} \ldots d\randomvar_{1} \big) \nonumber \\
 &+ \big( \int_{-1}^{1}\ldots\int_{-1}^{1} ( \int_{0}^{T} \int_{\Gamma} - 2 \mathcal{V} p \dot{q} + 2 \mathcal{K} \varphi \dot{q} + 2 \mathcal{K}' p \dot{\psi} - 2 \mathcal{W} \varphi \dot{\psi} ds_x dt ) \omega(\xi)  d\randomvar_{n} \ldots d\randomvar_{1} \big)
\end{align}
We use Legendre polynomials $\Psi_{i}$ as basis functions in the stochastic framework. In space
with $\xi_{h}^{s}$ resp. $\eta_{h}^{s}$ piecewise linear functions in space resp. piecewise constant functions in space and in time with $\beta^m_{\Delta t}(t)$ resp. $\gamma^m_{\Delta t}(t)$ piecewise linear functions (hat functions) in time resp. piecewise constant functions,
we choose ansatz functions as: 
\begin{align*}
  \varphi_{h,\Delta t}(t,x,\omega) &= \sum_{i=0}^{(\SGOrder+1)^{n}-1} \varphi_{i}(t,x) \Psi_{i}(\randomvar) \\
	p_{h,\Delta t}(t,x,\omega) &= \sum_{i=0}^{(\SGOrder+1)^{n}-1} p_{i}(t,x) \Psi_{i}(\randomvar) .
\end{align*}
with $ \varphi_{i} = \sum_{m=1}^{N_o} \sum_{s=1}^{N_{s}} \varphi_{i}^{m,s} \beta_{\Delta t}^{m}(t) \randomvar_{h}^{s}(x)$
and $ p_{i} = \sum_{m=1}^{N_o} \sum_{s=1}^{N_{s}} p_{i}^{m,s} \beta_{\Delta t}^{m}(t) \eta_{h}^{s}(x) $.

Further we choose test functions with $l=1,\ldots,N_{o}$, $r=1,\ldots,N_{s}$ and $j=0,\ldots,(\SGOrder+1)^{n}-1$
\begin{align*}
  \dot{\psi}_{h,\Delta t}(t,x,\omega) &= \gamma_{\Delta t}^{l}(t) \randomvar_{h}^{r}(x) \Psi_{j}(\randomvar) \\
	q_{h,\Delta t}(t,x,\omega) &= \gamma_{\Delta t}^{l}(t) \eta_{h}^{r}(x) \Psi_{j}(\randomvar) .
\end{align*}

We separate \eqref{eq:appendixacosutic} by exploiting the linearity of the integral operators.
\begin{align*}
& \sum_{i=0}^{(\SGOrder+1)^{n}-1} \int_{-1}^{1} \ldots \int_{-1}^{1} \int_{0}^{T} \int_{\Gamma} \tfrac{1}{\alpha(x,\xi)} p_{h,\Delta t}(t,x) \Psi_{i}(\randomvar) \gamma^{l}(t) \eta_{h}^{r}(x) \Psi_{j}(\randomvar) \omega(\randomvar)  ds_x dt d\randomvar_{n} \ldots \randomvar_{1} \\
&+ \sum_{i=0}^{(\SGOrder+1)^{n}-1} \int_{-1}^{1} \ldots \int_{-1}^{1} \int_{0}^{T} \int_{\Gamma} \alpha(x,\xi) \dot{\varphi}_{h,\Delta t}(t,x) \Psi_{i}(\randomvar) \gamma^{l}(t) \randomvar_{h}^{r}(x) \Psi_{j}(\randomvar) \omega(\randomvar) ds_x dt d\randomvar_{n} \ldots \randomvar_{1} \\
&+ \sum_{i=0}^{(\SGOrder+1)^{n}-1} \sum_{m=1}^{N_o} \sum_{s=1}^{N_{s}} p_{i}^{m,s} \big( \int_{-1}^{1} \ldots \int_{-1}^{1} \Psi_{i}(\randomvar) \Psi_{j}(\randomvar) \omega(\randomvar) d\randomvar_{n} d\randomvar_{1} \big) \\&\Big[ (-2) \big(\int_{0}^{T} \int_{\Gamma} \int_{\Gamma}  \frac{\beta_{\Delta t}^{m}(t-|x-y|) \eta_{h}^{s}(x) \dot{\gamma}_{\Delta t}^{l}(t) \eta_{h}^{r}(x)}{4 \pi |x-y|} ds_y ds_x dt \big) \\ &+ 2 \big( \int_{\Gamma} \int_{\Gamma} \frac{n_x (x-y)}{4 \pi |x-y|} ( \frac{\eta_{h}^{s}(y) \xi_{h}^{r}(x)}{|x-y|^{2}}\int_{0}^{T} \beta_{\Delta t}^{m}(t-|x-y|)  \gamma_{\Delta t}^{l}(t) dt \\&+ \frac{\eta_{h}^{s}(y) \xi_{h}^{r}(x)}{|x-y|} \int_{0}^{T} \dot{\beta}_{\Delta t}^{m}(t-|x-y|) \gamma_{\Delta t}^{l}(t) dt)  ds_y ds_x \big) \Big] \\
&+ \sum_{i=0}^{(\SGOrder+1)^{n}-1} \sum_{m=1}^{N_o} \sum_{s=1}^{N_{s}} \varphi_{i}^{m,s} \big( \int_{-1}^{1} \ldots \int_{-1}^{1} \Psi_{i}(\randomvar) \Psi_{j}(\randomvar) d\randomvar_{n} d\randomvar_{1} \big) \\ 
&\Big[2 \Big( \int_{\Gamma} \int_{\Gamma} \frac{n_y (x-y)}{4 \pi |x-y|} \big( \frac{\xi_{h}^{s}(y) \eta_{h}^{r}(x)}{|x-y|^{2}}\int_{0}^{T} \beta_{\Delta t}^{m}(t-|x-y|)  \dot{\gamma}_{\Delta t}^{l}(t) dt\\&
+ \frac{\xi_{h}^{s}(y) \eta_{h}^{r}(x)}{|x-y|} \int_{0}^{T} \dot{\beta}_{\Delta t}^{m}(t-|x-y|) \dot{\gamma}_{\Delta t}^{l}(t) dt \big) ds_y ds_x \Big) \\ 
&+ (-2) \Big( \int_{0}^{T} \int_{\Gamma} \int_{\Gamma} \frac{1}{4 \pi} \frac{-n_x \cdot n_y}{|x-y|} \big(\dot{\beta}_{\Delta t}^{m}(t-|x-y|) \randomvar_{h}^{s}(y) \dot{\gamma}_{\Delta t}^{l}(t) \randomvar_{h}^{r}(x) \\
&+ \frac{(curl_{\Gamma} \xi_{h}^{s})(y) (curl_{\Gamma} \xi_{h}^{r}(x))}{|x-y|} \beta_{\Delta t}^{m}(t-|x-y|) \gamma_{\Delta t}^{l}(t)  \big) ds_y ds_x dt \Big) \Big]
\end{align*}
The stochastic terms of the third and fourth terms can be handled as in the case for $\mathcal{V}$ the section before. 
Since the first term and the second term depend on $\alpha$, which may depend on the stochastic and the space, we have to be careful.
Consider the first and second term. Using the ansatz and test functions, we get for the time integral:
\begin{align*}
\sum_{m=1}^{N_o} p_{i}^{m,s} \int_{0}^{T} \beta_{\Delta t}^{m}(t) \gamma_{\Delta t}^{l}(t) dt = (p_{i}^{l,s}+p_{i}^{l-1,s}) \tfrac{\Delta t}{2} \\
\sum_{m=1}^{N_o} \varphi_{i}^{m,s} \int_{0}^{T} \dot{\beta}_{\Delta t}^{m}(t) \gamma_{\Delta t}^{l}(t) dt = (\varphi_{i}^{l,s}-\varphi_{i}^{l-1,s})
\end{align*}
Therefore
\begin{align*}
&\sum_{i=0}^{(\SGOrder+1)^{n}-1} \sum_{s=1}^{N_o} (p_{i}^{l,s}+p_{i}^{l-1,s}) \tfrac{\Delta t}{2} \int_{-1}^{1} \ldots \int_{-1}^{1} \int_{\Gamma} \tfrac{1}{\alpha} \eta_{h}^{s}(x) \Psi_{i}(\randomvar) \eta_{h}^{r}(x) \Psi_{j}(\randomvar) \omega(\randomvar) ds_x d\randomvar_{n} \ldots \randomvar_{1} \\
&\sum_{i=0}^{(\SGOrder+1)^{n}-1} \sum_{s=1}^{N_o} (\varphi_{i}^{l,s}-\varphi_{i}^{l-1,s}) \int_{-1}^{1} \ldots \int_{-1}^{1} \int_{\Gamma} \alpha \randomvar_{h}^{s}(x) \Psi_{i}(\randomvar) \randomvar_{h}^{r}(x) \Psi_{j}(\randomvar) \omega(\randomvar) ds_x d\randomvar_{n} \ldots \randomvar_{1}
\end{align*}
Assuming we may divide $\alpha(x,\omega)$ into $\alpha_{1}(x)\alpha_{2}(\omega)$, we seperate the integral further. Define:
\begin{align*}
 S_{i,j}(\tfrac{1}{\alpha_{2}}) &:= \int_{[-1,1]^{n}} \frac{1}{\alpha_{2}(\xi)} \Psi_{i}(\xi) \Psi_{j}(\xi) w(\xi) d\xi \\
I_{const}^{r,s}(\tfrac{1}{\alpha_{1}}) &= \int_{\Gamma}  \tfrac{1}{\alpha_{1}(x)} \eta_{h}^{s}(x) \eta_{h}^{r}(x) ds_x \\
I_{lin}^{r,s}(\tfrac{1}{\alpha_{1}}) &= \int_{\Gamma}  \alpha_{1}(x) \xi_{h}^{s}(x) \xi_{h}^{r}(x) ds_x 
\end{align*}
Then we get with $\otimes$ the Kronecker product the stochastic space matrices:
\begin{equation*}
  S(\tfrac{1}{\alpha_{2}}) \otimes I_{const}(\tfrac{1}{\alpha_{1}}) \quad \text{and}  \quad S(\alpha_{2}) \otimes I_{lin}(\alpha_{1}) .
\end{equation*}
For the computation of the boundary integral operators we refer to the Appendix of \cite{FSI} and \cite{review}. 

Altogether, we get a marching on in time scheme for the stochastic space system. Therefore for every time step $l$, solve the stochastic space system:
{\small{
\begin{align*}
  &\hspace*{-0.4cm}\begin{pmatrix} \frac{\Delta t}{2} S(\frac{1}{\alpha_{2}}) \otimes I_{const}(\frac{1}{\alpha_{1}}) + 2 diag(\frac{1}{(2 i_{n}+1)\ldots (2 i_{1}+1)} V^{0}) & 2 diag(\frac{1}{(2 i_{n}+1)\ldots (2 i_{1}+1)} K^{0}) \\ 2 diag(\frac{1}{(2 i_{n}+1)\ldots (2 i_{1}+1)} {K'}^{0}) & S(\alpha_{2}) \otimes I_{lin}(\alpha_{1}) - 2 diag(\frac{1}{(2 i_{n}+1)\ldots (2 i_{1}+1)} {W}^{0})  \end{pmatrix}
	\begin{pmatrix} p^{l} \\ \varphi^{l} \end{pmatrix} = \\ &\hspace*{-0.4cm}\begin{pmatrix} G - \frac{\Delta t}{2} S(\frac{1}{\alpha_{2}}) \otimes I_{const}(\frac{1}{\alpha_{1}}) p^{l-1} -2 \sum_{m=1}^{l-1} diag(\frac{1}{(2 i_{n}+1)\ldots (2 i_{1}+1)} V^{l-m}) p^{m}- 2 \sum_{m=1}^{l-1} diag(\frac{1}{(2 i_{n}+1)\ldots (2 i_{1}+1)} K^{l-m}) \varphi^{m} \\ 
	F + S(\alpha_{2}) \otimes I_{lin}(\alpha_{1}) \varphi^{l-1}  + 2 \sum_{m=1}^{l-1} diag(\frac{1}{(2 i_{n}+1)\ldots (2 i_{1}+1)} W^{l-m})  \varphi^{m} - 2 \sum_{m=1}^{l-1} diag(\frac{1}{(2 i_{n}+1)\ldots (2 i_{1}+1)} {K'}^{l-m}) p^{m} \end{pmatrix},
\end{align*} 
}}
where $V^{l-1},K^{l-1},{K'}^{l-1}, W^{l-1}$ are the corresponding matrices for the $l$-th time step.

\subsection{Discretization of the problem from traffic noise}\label{appendix:disctraffic}
{In this subsection, we discuss the details on the problem in subsection \ref{app:traficnoise}.
For a single layer potential ansatz, we apply the jump relations to derive from the acoustic boundary condition the integral equation:
\begin{equation*}
(-\textstyle{\frac{1}{2}} I + \mathcal{K}'p)- \alpha \textstyle{\frac{\partial }{\partial t}} \mathcal{V} p = f(x,t).
\end{equation*}
Therefore the variational formulation reads:
Find $p \in  L^{2}_{\omega}(\Xi;H^{1}([0,T],L^{2}(\Gamma)))$, such that for all $q \in L^{2}_{\omega}(\Xi;H^{1}([0,T],L^{2}(\Gamma)))$ 
\begin{align*}
\int_{\Xi} \int_{\Gamma} \int_{0}^{T} ((-\textstyle{\frac{1}{2}} I + \mathcal{K}')p - \alpha(x,t,\xi) \mathcal{V} \dot{p}) \dot{q} \omega(\xi) dt ds_x d\xi = \int_{\Xi} \int_{\Gamma} \int_{0}^{T} f(x,t,\xi) \dot{q} \omega(\xi) dt ds_x d\xi.
\end{align*}
Here again we use Legendre polynomials and piecewise constant ansatz functions in time resp. in space. For the test function, we again use $\dot{q}_{h,\Delta t}=\gamma_{\Delta t}^{m}(t) \eta_{h}^{r}(x) \Psi_{j}(\randomvar) $, where $\gamma_{\Delta t}^{m}(t)$ and $\eta_{h}^{r}(x)$ are piecewise constant functions in time, resp. in space, with $m=1,\ldots,N_o $, $r=1,\ldots,N_s$ and $j=0,\ldots,(J+1)^{n}-1$.
Therefore, we get:
\begin{align*}
\sum_{i=0}^{(J+1)^{n}-1} \int_{-1}^{1} \ldots \int_{-1}^{1} \int_{0}^{T} \int_{\Gamma} [(-\textstyle{\frac{1}{2}} I + \mathcal{K}')p_{h,\Delta t} - \alpha \mathcal{V} \dot{p}_{h,\Delta t} ] \dot{q}_{h,\Delta t} \omega(\randomvar) ds_x dt d\randomvar_{n}\ldots d\randomvar_{1}.
\end{align*}
For the computation of the single layer operator, we refer to the computation above in \ref{appendix:discstochdir} and for the adjoint boundary layer operator, we refer to the Appendix of \cite{FSI} and to \cite{review}. \\
As in the main text of the article, we obtain a marching-on-in-time scheme to solve the following system for the solution in each time step. \\If $\alpha$ only  depends on the stochastic variable, the individual operators factorize into tensor products with a fixed stochastic matrix: For every time step $l$, one then needs to solve the stochastic and spatial system:
\begin{align*}
&diag\left(\frac{1}{(2 i_{n}+1)\ldots (2 i_{1}+1)} \left(\frac{-\Delta t}{2}I_{const}(1)+{K'}^{0}\right)\right)-S(\alpha) \otimes V^{0} p^{l} \\
&= F^{l} - \sum_{m=1}^{l-1} diag\left(\frac{1}{(2 i_{n}+1)\ldots (2 i_{1}+1)} {K'}^{l-m}\right) p^{m} + \sum_{m=1}^{l-1} S(\alpha) \otimes V^{l-m} p^{m}
\end{align*}}

\end{document}